\documentclass[a4paper,10pt]{article}
\usepackage[utf8]{inputenc}

\usepackage{wrapfig}
\usepackage{amsmath} 
\usepackage{amsthm} 
\usepackage{amssymb} 
\usepackage{MnSymbol} 
\usepackage{enumerate} 
\usepackage{esint} 
\usepackage{pgf,tikz} 
\usetikzlibrary{arrows} 
 \usepackage{yfonts} 
 \usepackage{mathrsfs} 
 \usepackage{graphicx}
\usepackage{caption}
\usepackage{subcaption}
\usepackage{mathtools} 
\usepackage[titletoc,toc]{appendix} 

\definecolor{ffffff}{rgb}{1.0,1.0,1.0}
\definecolor{qqqqff}{rgb}{0.0,0.0,1.0}
\definecolor{ffqqqq}{rgb}{1.0,0.0,0.0}
\definecolor{zzzzqq}{rgb}{0.6,0.6,0.0}
\definecolor{marronet}{rgb}{0.6,0.2,0}
\definecolor{negre}{rgb}{0,0,0}
\definecolor{vermell}{rgb}{0.8,0.05,0.05}
\definecolor{blau}{rgb}{0.3,0.2,1.}
\definecolor{blauclar}{rgb}{0.,0.,1.}
\definecolor{grisfosc}{rgb}{0.25098039215686274,0.25098039215686274,0.25098039215686274}
\definecolor{verd}{rgb}{0.1,0.6,0.1}
\definecolor{taronja}{rgb}{0.9,0.6,0.05}
\definecolor{vermellclar}{rgb}{1.,0.,0.}
\definecolor{verdet}{rgb}{0,0.8,0.1}
\definecolor{blauverd}{rgb}{0,0.4,0.2}
\definecolor{grisclar}{rgb}{0.6274509803921569,0.6274509803921569,0.6274509803921569}

\usepackage{xcolor}

\newcommand{\R}{{\mathbb R}}       
\newcommand{\N}{{\mathbb N}}       

\newcommand{\HH}{{\mathbf H}}
\newcommand{\dist}{{\rm dist}}
\newcommand{\red}{{\rm red}}

\newcommand{\Beta}{{\beta}}

\newcommand{\rf}[1]{{(\ref{#1})}}
\newcommand{\supp}{{\rm supp}}

\newcommand{\norm}[1]{{\left\| {#1} \right\|}}

\DeclareMathOperator{\dive}{div}

\DeclareMathOperator{\loc}{loc}

\newtheorem{theorem}{Theorem}
\newtheorem*{theorem*}{Theorem}
\newtheorem{lemma}[theorem]{Lemma}

\newtheorem{corollary}[theorem]{Corollary}
\newtheorem*{corollary*}{Corollary}
\newtheorem{proposition}[theorem]{Proposition}

\newtheorem{definition}[theorem]{Definition}

\newtheorem{remark}[theorem]{Remark}

\numberwithin{subsection}{section}
\numberwithin{theorem}{section}
\numberwithin{equation}{section}
\numberwithin{figure}{section}

\usepackage[affil-it]{authblk}

\title{Minimizers for the thin one-phase free boundary problem}

\author{Max Engelstein \thanks{ME (University of Minnesota Twin Cities, Minnesota): \texttt{maxe@mit.edu}}\,,
Aapo Kauranen \thanks{AK (Universitat Aut\`onoma de Barcelona-BGSMath, Catalonia): \texttt{aapo.p.kauranen@jyu.fi}} \,,
Mart\'i Prats \thanks{MP (Universitat Aut\`onoma de Barcelona, Catalonia): \texttt{mprats@mat.uab.cat}} \,,
Georgios Sakellaris \thanks{GS (Universitat Aut\`onoma de Barcelona, Catalonia): \texttt{gsakellaris@mat.uab.cat}} \,,
Yannick Sire \thanks{YS (Johns Hopkins University, Maryland): \texttt{sire@math.jhu.edu}}
}

\begin{document}
\maketitle
\bibliographystyle{alpha}

\begin{abstract}
We consider the ``thin one-phase" free boundary problem, associated to minimizing a weighted Dirichlet energy of the function in $\mathbb R^{n+1}_+$ plus the area of the positivity set of that function in $\mathbb R^n$. We establish full regularity of the free boundary for dimensions $n \leq 2$, prove almost everywhere regularity of the free boundary in arbitrary dimension and provide content and structure estimates on the singular set of the free boundary when it exists. All of these results hold for the full range of the relevant weight. 

While our results are typical for the calculus of variations, our approach does not follow the standard one first introduced in \cite{AltCaffarelli}. Instead, the nonlocal nature of the distributional measure associated to a minimizer necessitates arguments which are less reliant on the underlying PDE. 
\end{abstract}


\setcounter{tocdepth}{2} 
\tableofcontents

\section{Introduction}\label{secIntro}

This article is devoted to the study of the regularity properties of a weighted version of the thin one-phase problem. More precisely we investigate even, nonnegative minimizers of the following functionals: denote $x\in\R^{n+1}$ by $x=(x',y)\in \R^n\times \R$, and for $\beta \in (-1,1)$ we define

\begin{equation}\label{eqLocalizedFunctional}
\mathcal{J}(v,\Omega):= \int_\Omega |y|^\beta |\nabla v|^2\ dx + m(\{v>0\}\cap\R^n\cap\Omega),
\end{equation}
where $m$ stands for the $n$-dimensional Lebesgue measure. 
Here, and throughout the paper, the integration is done with respect to the $(n+1)$-dimensional Lebesgue measure unless stated otherwise.
This functional is finite for open sets, $\Omega$, and functions in the weighted Hilbert space,
$$H^1(\beta,\Omega): = \{ v\in L^2(\Omega;|y|^\beta) : \nabla v\in L^2(\Omega;|y|^\beta)\},$$
equipped with the usual weighted norm.

Our main concern is to investigate fine regularity properties of the free boundary of minimizers $v$ of \eqref{eqLocalizedFunctional}, that is the set,
$$
F(v):=\partial_{\mathbb R^n} \left \{ v(x,0)>0 \right \} \cap \Omega. 
$$
Since the free boundary lies on a codimension 1 subspace of the ambient space $\R^{n+1}$, such a problem is called a {\sl thin one-phase free boundary problem}. This type of free boundary problem has been investigated for the first time by Caffarelli, Roquejoffre and the last author in \cite{CaffarelliRoquejoffreSire} in relation with the theory of semi-permeable membranes (see, e.g., \cite{duvautLions}). As we will describe later this is an analogue of the classical one-phase problem (also called the Bernoulli problem) but for the fractional Laplacian.

The Bernoulli problem was first treated in a rigorous mathematical way by Alt and Caffarelli in the seminal paper \cite{AltCaffarelli}: in the Bernoulli problem we consider minimizers of \eqref{eqLocalizedFunctional} where $\beta = 0$ and the second term  is replaced by $\mathcal L^{n+1}(\{v>0\}\cap\Omega)$ (where $\mathcal L^{n+1}$ stands for the Lebesgue measure in $\R^{n+1}$). In particular, for the Bernoulli problem, the free boundary fully sits in the ambient space, $\R^{n+1}$. In \cite{AltCaffarelli}, the authors provided a general strategy to attack this type of problem.  Out of necessity we needed to modify this blueprint in several substantial ways (see below for a more detailed comparison).  For more information on the one-phase problem (and some of its variants) we refer to the book of Caffarelli and Salsa (and  references therein) \cite{CaffarelliSalsa}, and to the more recent survey of De Silva, Ferrari and Salsa \cite{DeSilvaFerrariSalsaSurvey}. 

As noticed in \cite{CaffarelliRoquejoffreSire}, problem \eqref{eqLocalizedFunctional} is related in a tight way to the {\sl standard } one-phase free boundary problem but with the Dirichlet energy replaced by the Gagliardo semi-norm $[u]_{\dot H^\alpha}$, for $\alpha=\frac{1-\beta}{2} \in (0,1)$. This connection suggests that the thin one-phase problem is actually intrinsically a {\sl nonlocal } problem, though the energy in \eqref{eqLocalizedFunctional} is clearly local. 
\subsection*{Connection with the fractional one-phase problem}
As previously mentioned, the functional $\mathcal{J}$  introduced by Caffarelli, Roquejoffre and the last author in \cite{CaffarelliRoquejoffreSire} is a local version of the following nonlocal free boundary problem: given a function $f\in L^1_{\loc}(\R^n)$ with suitable decay at infinity, we can define its fractional Laplacian at $x\in\R^n$ by
$$(-\Delta)^\alpha f(x) = c_{n,\alpha}\, p.v. \int_{\R^n} \frac{f(x)-f(\xi)}{|x-\xi|^{n+2\alpha}}\, d\xi.$$
At the formal level, we are interested in solutions of the free boundary problem  
\begin{equation}\label{eqEulerLagrangeFrac}
\begin{cases}
(-\Delta)^\alpha f = 0  & \mbox{in } \Omega \cap \{f>0\},\\
\partial_\nu^\alpha f = A  & \mbox{on } \Omega\cap F(f),
\end{cases}
\end{equation}
where $\partial_\nu^\alpha f (x):=\lim_{ \Omega \cap \{f>0\} \ni \xi\to x}\frac{f(\xi)-f(x)}{((\xi-x)\cdot \nu(x))^\alpha}$ and where $f$ satisfies a given ``Dirichlet boundary condition" on the complement of $\Omega$. 


As in the case of the classical Laplacian (see \cite{AltCaffarelli}), we are interested in obtaining equation \rf{eqEulerLagrangeFrac} as the Euler-Lagrange equation of a certain functional. Given a locally integrable function $f$, consider its fractional Sobolev energy
$$\left[f\right]_{\dot H^\alpha (\R^n)}:=\iint_{\R^{2n}} \frac{|f(x)-f(\xi)|^2}{|x-\xi|^{n+2\alpha}}\, d\xi\, dx.$$
Since we want to study competitors which vary only in a certain domain $\Omega$, it is natural to consider only the integration region which may suffer variations when changing candidates. Thus, we define the energy 
\begin{equation}\label{eqNonLocalFunctional}
J(f,\Omega):= c_{n,\alpha} \iint_{\R^{2n}\setminus (\Omega^c)^2} \frac{|f(x)-f(\xi)|^2}{|x-\xi|^{n+2\alpha}}\, d\xi\, dx + m(\{f>0\}\cap\Omega).
\end{equation}
We say that $f\in L^1_{\loc}$ is a minimizer of $J$ in $\Omega$ if $J(f,\Omega)$ is finite and $ J(f,\Omega)\leq J(g,\Omega)$ for every $g$ satisfying that $f-g\in \dot H^\alpha(\R^n)$ and such that $f(x)=g(x)$ for almost every $x\in \Omega^c$. We say that $f$ is a global minimizer if it is a minimizer for every open set $\Omega\subset \R^n$. Note that both terms in \rf{eqNonLocalFunctional} are in competition, since a minimizer of the fractional Sobolev energy in $\Omega$ is $\alpha$-harmonic and, thus, if it is non-negative outside of $\Omega$ it is strictly positive inside of $\Omega$, maximizing the second term.



Consider now the Poisson kernel for fixed $n\in\N$ and $0<\alpha<1$
\begin{equation}\label{eqPoisson}
P_y(\xi):=P_{n,\alpha}(\xi,y)=c_{n,\alpha}  \frac{|y|^{2\alpha}}{|(\xi,y)|^{n+2\alpha}} \quad\quad \mbox{ for every }(\xi,y)\in \R^n\times\R.
\end{equation}
The Poisson extension of $f\in L^1_{\loc}(\R^n)$ is given by
\begin{equation}\label{eqPoissonExtension}
u(x',y):= f * P_y(x')=\int_{\R^n} P_{n,\alpha}(\xi,y) f(x'-\xi)\, d\xi \quad\quad \mbox{ for every }(x',y)\in \R^n\times\R.
\end{equation}
By \cite{CaffarelliSilvestre}, with a convenient choice of the constant one gets
$$\lim_{y \searrow 0}y^{1-2\alpha} u_y(x',y)=-(-\Delta)^\alpha f(x')$$ in every point where $f$ is regular enough. Moreover, the extension satisfies the localized equation $\nabla\cdot (|y|^\beta \nabla u)=0$ weakly, away from $\R^n\times \{0\}$. The whole point is that local minimizers of \eqref{eqNonLocalFunctional} can be extended via the previous Poisson kernel $P_y$ to (even) minimizers of \eqref{eqLocalizedFunctional} (see the Appendix for a precise statement). Therefore, the thin one-phase problem appears as a ``localization" of the one-phase problem for the fractional Laplacian. Notice that, and this is of major importance for us, this localization technique does not carry over to other types of nonlocal operators besides pure powers of second-order elliptic operators. This is a major drawback of the theory, in the sense that, at the moment, it seems to be impossible to tackle one-phase problems involving more general operators than the fractional Laplacian. The main point is we do not know how to prove any kind of monotonicity for general integral operators.

This connection between the nonlocal analogue of the Bernoulli problem and our thin one-phase problem allows us to simplify several arguments by working in the purely nonlocal setting. However, this underlying nonlocality is also the reason why several results, which came more easily in the setting of \cite{AltCaffarelli}, are non-trivial or substantially harder for us. For example, perturbations of solutions need to take into account long range effects which makes classical, local, perturbation arguments much more difficult. 



In the paper \cite{CaffarelliRoquejoffreSire}, the authors proved basic properties of the minimizers for the functional $\mathcal J$ such as optimal regularity, non-degeneracy near the free boundary, and  positive densities of phases. Also they provided an argument for $n=2$ showing that Lipschitz free boundaries are $C^1$. A feature of the functional $\mathcal J$ is that the weight $|y|^\beta$ is either degenerate or singular at $\left \{ y=0 \right \}$ (except in the case $\beta=0$). Such weights belong to the Muckenhoupt class $A_2$ and the seminal paper of Fabes, Kenig and Serapioni \cite{FabesKenigSerapioni} investigated regularity issues for elliptic PDEs involving such weights (among other things). After that, \cite{deSilvaSavinSire} proved an $\varepsilon$-regularity result and \cite{Allen} showed the existence of a monotonicity formula for this setting.

In the case $\beta=0$, the problem is still degenerate in the sense that derivatives near the free boundary blow up. The case $\beta=0$ has been thoroughly investigated in the series of papers by De Silva, Savin  and Roquejoffre \cite{desilvaRoque,DSSJDE,deSilvaSavinJEMS}. 

The main goal of our paper is to provide a full picture of the regularity of the free boundary for any power $\beta \in (-1,1)$, both in terms of measure-theoretic statements and partial (or full) regularity results. From this point of view our contribution is a complement of the paper by De Silva and Savin \cite{deSilvaSavinJEMS} for $\beta=0$. It has to be noticed that the standard approach to regularity of Lipschitz  free boundaries as developed by Caffarelli (see the monograph \cite{CaffarelliSalsa}) does not seem to work in our setting.

\subsubsection*{Our approach to regularity}

In \cite{AltCaffarelli} (and many subsequent works), the minimizing property of the solution is used to prove that the distributional Laplacian of that solution is an Ahlfors-regular measure supported on the free boundary. This implies (amongst other things) that the free boundary is a set of (locally) finite perimeter, and thus almost every point on the free boundary has a measure theoretic tangent. One can then work purely with the weak formula (i.e. the analogue of \eqref{eqEulerLagrangeFrac}) to prove a ``flat implies smooth" result which, together with the existence almost everywhere of a measure theoretic tangent, has as a consequence that the free boundary is almost everywhere a smooth graph and the free boundary condition in \eqref{eqEulerLagrangeFrac} holds in a classical sense at the smooth points. 

A similar ``flat implies smooth" result exists in our context (this is essentially due to De Silva, Savin and the last author, \cite{deSilvaSavinSire}, see Theorem \ref{theoImprovement} below). However, showing that the free boundary is the boundary of a set of finite perimeter proves to be much more difficult. Due to the nonlocal nature of the problem, $-\mathrm{div}(|y|^\beta \nabla u)$ (considered as a distribution) is not supported on the free boundary. Furthermore, the scaling of this measure does not allow us to conclude that the free boundary has the correct dimension (much less that it is Ahlfors regular).



To prove finite perimeter, we take the following approach inspired by the work of de Silva and Savin: after establishing some preliminaries 
we prove crucial compactness results. This, along with a monotonicity formula originally due to Allen \cite{Allen}  allows us to run a dimension reduction argument in the vein of Federer or (in the context of free boundary problems) Weiss \cite{WeissMinimum}. With this tool in hand, we show that the set of points at which no blow-up is flat is a set of lower dimension. Locally finite perimeter and regularity for the reduced boundary then follow from a covering argument and some standard techniques.

 Here and throughout  the paper, we will denote the ball of radius $r$ in $\R^{n+1}$ centered at the origin by  $B_r$, and $B_r':=B_r\cap \R^n\times\{0\}$. Moreover, for the definition of $\HH^\beta$, see Section \ref{s:preliminaries}. We may then summarize our regularity results in the following theorem.

\begin{theorem}\label{t:mainregularity}[Main Regularity Theorem]
Let $u\in \HH^\beta(B_1)$ be a (non-negative, even) local minimizer of $\mathcal{J}$ in $B_1 \subset \mathbb R^{n+1}$. Let $B_{1,+}'(u):=\{x=(x',0)\in B_1: u(x)>0\}$, let $F(u)$ be the boundary of $B_{1,+}'(u)$ inside of $\R^n\times \{0\}$  and assume that $0\in F(u)$. Then,
\begin{enumerate}
\item $B_{1,+}'(u)$ (as a subset of $\R^n\times \{0\}$) is a set of locally finite perimeter in $B_1'$. 
\item We can write the free boundary as a disjoint union $F(u) = \mathcal R(u) \cup \Sigma(u)$, where $\mathcal R(u)$ is open inside $F(u)$, and for $x \in \mathcal R(u)$ there exists an $r_x > 0$ such that $B(x, r_x) \cap F(u)$ can be written as the graph of a $C^{1,s}$-continuous function. 
\item Furthermore, the set $\Sigma(u)$ is of Hausdorff dimension $\leq n-3$ (and, therefore, of $\mathcal H^{n-1}$-measure zero). In particular, for $n\leq 2$, $\Sigma(u)$ is empty, and moreover, if $n=3$ then $\Sigma(u)$ is discrete.
\end{enumerate}
The constants (implicit in the set of finite perimeter, and the H\"older continuity of the functions whose graph gives the free boundary) depend on $n$ and $\beta$ but not on $\norm{u}_{\HH^\beta(B_1)}$. 
\end{theorem}

As usual $\Sigma(u) \subset F(u)$ is called the \emph{singular set} of the free boundary: the set of points around which $F(u)$ cannot be parameterized as a smooth graph and all the blow-ups will be non-trivial minimal cones, see Theorem \ref{theoImprovement}.

Our second contribution concerns the structure and size of the singular set. It builds on recent major works on quantitative stratification \cite{NaberValtortaRectifiable}, extended to free boundary problems (in particular the one-phase problem) by Edelen and the first author \cite{EdelenEngelstein}. 

\begin{theorem}\label{t:singset}
Let $u\in \HH^\beta(B_1)$ be a (non-negative, even) local minimizer of $\mathcal{J}$ in $B_1$ and $0\in F(u)$. Let $B_{1,+}'(u):=\{x=(x',0)\in B_1: u(x)>0\}$ and $F(u)$ be the boundary of $B_{1,+}'(u)$ inside  $B_1'$. Then, there exists a $ k_\alpha^*\geq 3$  such that $\Sigma(u)$ is $(n-k_\alpha^*)$-rectifiable and $$\mathcal{H}^{n-k^*_\alpha}(\Sigma(u)\cap D)\leq C_{n,\alpha, \dist(D,\partial B_1)} \quad\quad \mbox{for every }D\subset\subset B_1.$$

\end{theorem}

In \cite{DSJ}, De Silva and Jerison constructed a singular minimizer for the Alt-Caffarelli one-phase problem in dimension $7$, giving the dimension bound $k^* \leq 8$ in the previous theorem in this case (see \cite{EdelenEngelstein}). This result is not known for the thin one-phase problem. The reason is that the one-phase problem, seen from the nonlocal point of view involving the fractional Laplacian, is related to the so-called nonlocal minimal surfaces introduced by Caffarelli, Roquejoffre and Savin \cite{CaffarelliRoquejoffreSavin}. Indeed, in \cite{savinValdinoci}, the authors proved that a fractional version of Allen-Cahn equation converges variationally to the standard perimeter functional for $\alpha \geq 1/2$ and to the so-called nonlocal minimal surfaces for $\alpha <1/2$. We can then conjecture the bound $k^*_\alpha \leq 8$ for $\alpha \geq 1/2$ by analogy with the result for the standard one-phase problem but the bound for $\alpha <1/2$ is not clear at all. However, one knows that there is no singular cone in dimension $2$ for nonlocal minimal surfaces \cite{SVcone} and that the Bernstein problem is known for those in dimensions 2 and 3 \cite{figalliValdinoci}.

We would like also to make a last remark about a result which is of purely nonlocal nature. In the case of the one-phase problem, one can show that the distributional Laplacian is a Radon measure along the free boundary. In the case of the thin one-phase free boundary problem, due to the nonlocality of the problem, such a behavior does not happen in the sense that we will show that the fractional Laplacian  is an absolutely continuous  measure with respect to $n$-dimensional Lebesgue measure with a precise behavior. This phenomenon is of purely nonlocal nature and similar to the fact that the fractional harmonic measure is of trivial nature.  More precisely, every minimizer $u$ satisfies $\nabla\cdot (|y|^\beta \nabla u)=0$ weakly, away from $\R^n\cap \{u \leq 0\}$. Thus, equation \rf{eqEulerLagrangeFrac} above can be understood as an Euler-Lagrange equation for the functional $\mathcal{J}$ in the sense that the restriction to $\R^n$ of a given minimizer $u$ in $\Omega\subset \R^{n+1}$ and with asymptotic behavior $u(x,y)=\mathcal{O}(|(x,y)|^\alpha)$ is always a solution to \rf{eqEulerLagrangeFrac} for $A=A(\alpha)$ at ``nice" points of the free boundary.

A brief summary of this paper  follows. In Sections \ref{s:Compactness} and  \ref{s:monotonicity} we discuss  compactness of minimizers and we recall Allen's monotonicity formula to derive some immediate consequences. In Section \ref{s:finperimeter} we show that the positive phase is a set of locally finite perimeter, establishing the first part of Theorem \ref{t:mainregularity} (modulo energy bounds), and we show that the singular set can be identified using the Allen-Weiss density. Section \ref{s:r2} is devoted to deducing full regularity of minimizers in $\R^{2+1}$ concluding the proof of Theorem \ref{t:mainregularity}.

Once we have established the finite perimeter, in Section \ref{s:UniformBounds} we remove the dependence of the estimates on the energy of the minimizer in the previous theorems, using a rather subtle argument which combines results from all the previous sections. A crucial step is to analyze some basic properties of the distributional fractional Laplacian of our minimizer. As stated above this analysis will not be enough to establish that the positivity set of the minimizer is a set of locally finite perimeter.  We believe that many of these results may be of independent interest. For example, corresponding results for the classical Bernoulli problem have been used to understand free boundary problems for harmonic measure (see \cite{KenigToroDCDS}).

Finally, Section \ref{s:Rect} is devoted to the proof of Theorem \ref{t:singset}.

\subsubsection*{Notation}

We denote the constants that depend on the dimension $n$, $\alpha$ and perhaps some other fixed parameters which are clear from the context  by $C$. Their value may change from an occurrence to another. On the other hand, constants with subscripts as $C_0$ retain their values along the text. For $a,b\geq 0$, we write $a\lesssim b$ if there is $C>0$ such that $a\leq Cb$. We write $a\approx b$ to mean $a\lesssim b\lesssim a$.

Let $u$ be a continuous function in $\R^{n+1}$. Then we write $\Omega_+(u):=\Omega\cap \{u>0\}$, and we denote the zero phase, the positive phase and the free boundary by
\begin{align*}
&\Omega_0(u):=\{x\in \R^n\times\{0\}: u(x)=0\}^\circ{},\\
&\Omega_+'(u):=\Omega_+\cap (\R^n\times\{0\})=\{x\in \R^n\times\{0\}: u(x)>0\},\mbox{ and}\\
&F(u):=F_\Omega(u)=\partial (\Omega_+(u)\cap \R^n\times\{0\}) \cap \Omega,
\end{align*}
respectively. Here both the boundary and the interior are taken with respect to the standard topology in $\R^n$.
Note that $\R^n\times \{0\}$ is the disjoint union of  $\Omega_0(u)$, $\Omega_+'(u)$ and $F(u)$ whenever $u$ is non-negative. We also call $F_{\rm red}(u)=F_{{\rm red},\Omega}(u)$ the points of $F_\Omega(u)$ where the free boundary is expressed locally as a $C^1$ surface. Finally, let $\Sigma(u)=\Sigma_\Omega(u)=F_\Omega(u)\setminus F_{{\rm red},\Omega}(u)$. In general we will write $\Omega':=\Omega\cap(\R^n\times\{0\})$.

Throughout the paper we will often fix $\beta \in (-1, 1)$ but then refer to $\alpha\in (0,1)$ or vice versa. These two numbers are always connected by the relationship $\alpha = \frac{1-\beta}{2}$.

%

\section{Preliminaries}\label{s:preliminaries}

In this section, we provide the known results concerning the problem under consideration. We say that a function $u$ is \emph{even} if it is symmetric with respect to the hyperplane $\R^n\times \{0\}$, that is, $u(x',y)=u(x',-y)$. The function spaces that we will consider are the following
$$ \HH^\beta (\Omega):= \{u \in H^1(\beta,\Omega) : u\mbox{ is even and non-negative}\}$$
and
$$ \HH^\beta_{\loc}(\Omega):= \{u\in L^2_{\loc}(\Omega): u \in \HH^\beta(B) \mbox{ for every ball } B\subset \subset \Omega\}.$$
We will omit $\Omega$ in the notation when it is clear from the context.
\begin{definition}\label{defMinimizer}
We say that a function $u\in \HH^\beta_{\loc}(\Omega)$ is a (local) minimizer of $\mathcal{J}$ in a domain $\Omega$ if for every ball $B\subset\subset \Omega$ and for every function $v\in \HH^\beta(B)$ such that the traces $v|_{\partial B} \equiv u|_{\partial B}$, the inequality
$$\mathcal{J}(u,B)\leq \mathcal{J}(v,B)$$
holds.
%
\end{definition}

As usual for several free boundary problems, it is a natural question to exhibit a particular (global) solution so that one gets an idea of the qualitative properties of general solutions. Let us consider the following function: for every $x\in \R^n$ let $$f_{n,\alpha}(x):= c_{n,\alpha} (x_n)_+^\alpha ,$$
where $a_+=\max\{0,a\}$.
If $n=1$, $f_{1,\alpha}$ is a solution to \rf{eqEulerLagrangeFrac} for a convenient choice of $c_{1,\alpha}$ (see \cite[Theorem 3.1.4]{BucurValdinoci}). In fact one can see that the same is true for $n\geq 1$ using Fubini's Theorem conveniently, with 
\begin{equation}\label{eqLaplacianTrivial}
-(-\Delta)^\alpha f_{n,\alpha}(x) = c_{n,\alpha} (x_n)_-^{-\alpha},
\end{equation}
where $a_-=\max\{0,-a\}$.

As a  toy question we wonder whether the trivial solutions are minimizers. Indeed, this is the case, as we will see later in Section \ref{secReduction}.
\begin{proposition}\label{propoTrivial}
Let $n\in\N$ and $0<\alpha<1$. Then the trivial solution $u_{n,\alpha}:=f_{n,\alpha}*P_y$ is a minimizer of $\mathcal{J}$ in every ball $B\subset \R^{n+1}$.
\end{proposition}

Next we collect the main properties of minimizers in the unit ball proven in \cite[Theorems 1.1-1.4, Proposition 3.3 and Corollary 3.4]{CaffarelliRoquejoffreSire}. 
\begin{theorem}\label{theoCRS}
If $u\in \HH^\beta(B_1)$ is a minimizer of $\mathcal{J}$ in $B_1$  with $\norm{u}_{\dot \HH^\beta(B_1)}:=\norm{\nabla u}_{L^2(B_1,|y|^\beta)} \leq E_0$ and $x_0\in F(u)\cap B_{\frac12}$, then it satisfies
\begin{enumerate}[P1:]
\item Optimal regularity (see \cite[Theorem 1.1]{CaffarelliRoquejoffreSire}): $\norm{u}_{\dot C^\alpha(B_{1/2})}\leq C{(1+E_0)}$.
\item Nondegeneracy (see \cite[Theorem 1.2]{CaffarelliRoquejoffreSire}): $u(x)\geq C \dist(x,F(u))^\alpha $ for $x\in B_{\frac12}'$.
\item Interior corkscrew condition (see \cite[Proposition 3.3]{CaffarelliRoquejoffreSire}):  there exists $x_+\in B_r'(x_0)$ so that  $B'(x_+, C_0r)\subset \Omega_+'(u)$.
\item Positive density (see \cite[Theorem 1.3]{CaffarelliRoquejoffreSire}): $|\Omega_0\cap B_r'(x_0)| \gtrsim r^n$. \label{itemPositive}
\item Blow-ups are minimizers (see \cite[Corollary 3.4]{CaffarelliRoquejoffreSire}): The limit of a blow-up sequence $u_k(x):=\frac{u(x_0+\rho_k x)}{\rho_k^\alpha}$ converging weakly in $H^1(\beta,B_1)$ and uniformly is a global minimizer.\label{itemBlowup}
\item Normal behavior at the free boundary (see \cite[Theorem 1.4]{CaffarelliRoquejoffreSire}): the boundary condition in \rf{eqEulerLagrangeFrac} is satisfied at every point on the free boundary with a measure theoretic normal (see \cite{EvansGariepy}) for a prescribed value of $A$.\label{itemNormal}
\end{enumerate}
All the constants depend on $n$ and $\alpha$; and also on  $E_0$ except for the one in \emph{P1}.
\end{theorem}

A major tool in the present paper is an $\epsilon-$regularity result, i.e. in the language of free boundaries a statement of the type ``flatness implies smoothness". In \cite{deSilvaSavinSire}, the authors proved such an $\epsilon$-regularity result for viscosity solutions to the overdetermined system associated to minimizers of $\mathcal J$. Here we establish that all local minimizers are in fact viscosity solutions. While this verification may be standard for experts in the field, we include it here for the sake of completeness. 

\begin{theorem}[$\epsilon$-regularity]\label{theoImprovement}
\label{th:iof}
There exists $\epsilon>0$ depending only on $n$, $\alpha$ and $E_0$ such that for every non-negative, even minimizer $u$ of the energy \rf{eqLocalizedFunctional} on a ball $B\subset \R^{n+1}$ with 
$\norm{u}_{\HH^\beta(B)}\leq  E_0 r(B)^{\frac n2}$ and 
\begin{equation}\label{eqImprovementOfFlatness}
\{(x,0)\in B: x_n\leq -\epsilon\}\subset B_0(u) \subset \{(x,0)\in B: x_n\leq \epsilon\},
\end{equation}
we have that $F(u)\in C^{1,\gamma}_{\rm loc}(\frac12 B)$, with $0<\gamma<1$.
\end{theorem}
Note that the dependence on $E_0$ will be removed in Section \ref{s:UniformBounds}.
\begin{proof}
 We say that $u$ is a \emph{viscosity solution} to 
\begin{equation}\label{eqViscositySol}
\begin{cases}
\nabla\cdot (|y|^\beta \nabla u) = 0 & \mbox{in }B_1^+(u),\\
\lim_{t\to 0+} \frac{u(x_0+t\nu (x_0),0)}{t^\alpha} = 1, & \mbox{for }  (x_0,0)\in F(u),
\end{cases}
\end{equation}
if 
\begin{enumerate}[i)]
\item $u\in C(B_1)$, $u\geq 0$,
\item $u\in C^{1,1}_{\rm loc}(B_{1,+}(u))$, $u$ is even and it solves $\nabla\cdot (|y|^\beta \nabla u) = 0$ in the viscosity sense, and
\item any strict \emph{comparison subsolution} (resp. supersolution) cannot touch from below (resp. from above) at a point $(x_0,0)\in F(u)$. 
\end{enumerate}

We claim that 
\begin{equation}\label{eqMinimizersViscosity}
\mbox{every non-negative even minimizer is a viscosity solution.}
\end{equation}

Conditions (i) and (ii) have been verified  in \cite{deSilvaSavinSire,VitaThesis}. To verify our claim it suffices to prove condition (iii) above: that any strict comparison subsolution cannot touch $u$ from below at a point $(x_0, 0) \in F(u)$. The analogous claim for strict comparison supersolutions will follow in the same way. 

Let us recall (see, e.g. Definition 2.2 in \cite{deSilvaSavinSire}),  that $w\in C(B_1)$ is a strict comparison subsolution (resp. supersolution) to \eqref{eqViscositySol} if 
\begin{enumerate}[a)]
\item $w \geq 0$,
\item $w$ is even with respect to $\{y= 0\}$, 
\item $w\in C^2(\{w > 0\})$,
\item $\mathrm{div}\left(|z|^\beta \nabla w\right) \geq 0$ in $B_1 \backslash \{y = 0\}$,
\item $F(w)$ is locally given by the graph of a $C^2$ function and for any $x_0 \in F(w)$ we may write \begin{equation}\label{e:exp} w(x,y) = aU((x-x_0)\cdot \nu(x_0), y) + o(\|(x-x_0, y)\|^\alpha),\qquad (x,y) \rightarrow (x_0,0).\end{equation} Here $U$ is the extension of the trivial solution (see \cite{deSilvaSavinSire}), $\nu(x_0)$ is the unit normal to $F(w)$ considered as a subset of $\R^n$ pointing into $\{w > 0\}$ and $a \geq 1$.
\item Furthermore, either the inequality is strict in d), or $a >1$ in e). 
\end{enumerate}

So assume that $w \geq u$ where $w$ is a strict comparison subsolution and $u$ is some minimizer and that $w = u$ at $(x_0,0) \in F(u)$. Since $u(x_0, 0) = 0$ it follows that $(x_0, 0) \in F(w)$ and with a harmless rotation we can guarantee that $\nu((x_0,0)) = e_n$. We want to show that $e_n$ is also the measure theoretic unit normal to $F(u)$. Indeed, since $F(w)$ is $C^2$ there must exist a ball $B \subset \{w > 0\}$ which is tangent to $F(w)$ at $(x_0,0)$. It must then be that case that $B \subset \{u > 0\}$ as well. Thus $(x_0,0) \in F(u)$ has a tangent ball from the inside which, by \cite{CaffarelliRoquejoffreSire} Proposition 4.5 implies that $u$ has the asymptotic expansion $$u(x,y) = U((x-x_0)\cdot \nu(x_0), y) + o(\|(x-x_0, y)\|^\alpha),\qquad (x,y) \rightarrow (x_0,0).$$ 

If $u \geq w$ this implies that $w$ must satisfy the expansion in \eqref{e:exp} with $a = 1$ at the point $x_0$. This, in turn, implies that $\mathrm{div}\left(|z|^\beta \nabla w\right) > 0$ in $B_1 \backslash \{y = 0\}$ (by the definition of a strict subsolution). Furthermore, since $w\in C^2$ where $\{w > 0\}$ we can guarantee that $\mathrm{div}\left(|z|^\beta \nabla w\right) \geq 0$ in all of $B_1 \cap \{w > 0\}$. 

Let us return to the ball $B$ which is a subset of $\{u > 0\}$ and $\{w > 0\}$ and for which $(x_0, 0) \in \overline{B}$. We know that $w-u \neq 0$ in $B\setminus\{y=0\}$ (this is because $w$ strictly satisfies the differential inequality in $B$ away from $\{y = 0\}$) and we know that $w-u$ is a subsolution in $B$. Furthermore $(x_0, 0) \in B$ is a strict maximum, so by the Hopf lemma in \cite[Proposition 4.11]{cabreSire} it must be that $$\lim_{t\downarrow 0^+}\frac{({w-u})(x_0+t\nu (x_0),0)}{t^\alpha} > 0.$$ This contradicts the fact $u$ and $w$ both satisfy \eqref{e:exp} at $(x_0,0)$ with $a = 1$.  Therefore, $(x_0,0)$ must not have been a touching point and $u$ is indeed a viscosity solution. 

Since, $u$ is a viscosity solution, \cite[Theorem 1.1]{deSilvaSavinSire} applies and we get the desired $\varepsilon$-regularity. 
\end{proof}

\section{Compactness of minimizers}\label{s:Compactness}

In this section we prove important results on the compactness of minimizers. As we mentioned above, our contribution is that convergent sequences of minimizers also converge in the relevant weighted Sobolev spaces strongly rather than just weakly. This will prove essential to the compactness arguments used in the later sections of this paper.

\subsection{Caccioppoli Inequality}\label{secCaccioppoli}
First we want to show that  the distribution $\lambda:=\nabla\cdot(|y|^\beta \nabla u)$ is in fact a Radon measure with support in the complement of the positive phase as long as $u$ is a minimizer. In Section \ref{s:UniformBounds} we will come back to this measure to understand its behavior around the free boundary.

\begin{lemma}\label{lemRadonMeasure}
Let $\Omega\subset \R^{n+1}$ be an open set, and let $u\in W^{1,2}_{\loc}(\Omega, |y|^\beta)$ be such that $\nabla\cdot (|y|^\beta \nabla u)=0$ weakly in $\Omega_+(u)$, i.e., for every $\eta\in C^\infty_c(\Omega_+(u))$,
\begin{equation}\label{eqBetaHarmonic}
\langle \nabla\cdot(|y|^\beta \nabla u), \eta\rangle := - \int (|y|^\beta \nabla u) \nabla \eta=0.
\end{equation}
Then $\lambda:=\nabla\cdot(|y|^\beta \nabla u)$ is a positive Radon measure supported on $\{u=0\}$  and for every $v\in W^{1,2}(\Omega, |y|^\beta) \cap C_{c}(\Omega)$ 
\begin{equation}\label{eqMeasureDef}
\int v \, d\lambda = -\int |y|^\beta \nabla u\cdot \nabla v.
\end{equation}
\end{lemma}
\begin{proof}
Indeed, by \rf{eqBetaHarmonic} the quantity
$$-\int |y|^\beta \nabla u\cdot \nabla \zeta=-\int |y|^\beta \nabla u \cdot \nabla \left(\zeta \max\left\{\min\left\{2-\frac{u}{\varepsilon},1\right\},0\right\}\right) \geq -\int_{\Omega\cap \{0<u<2\varepsilon\}} |y|^\beta |\nabla u| |\nabla \zeta|\xrightarrow{\varepsilon\to0} 0$$
defines a positive functional on positive  $\zeta\in C^{0,1}_c(\Omega)$. Moreover, for compact $K\subset \Omega$, consider a Lipschitz function $f_K$ such that $\chi_K\leq f_K\leq \chi_\Omega$. If $\zeta\in C^{0,1}_c(K)$, by the positivity shown above we obtain
$$-\int |y|^\beta \nabla u\cdot \nabla \zeta\leq - \norm{\zeta}_{L^\infty}\int |y|^\beta \nabla u\cdot \nabla f_K\leq C_{K,u} \norm{\zeta}_{L^\infty} $$
and, by Hahn-Banach's theorem, we can extend the functional to a positive functional in $C_c(\Omega)$, that is given by integration against a positive Radon measure by the Riesz representation theorem.

The fact  that \rf{eqMeasureDef} holds for all functions in $W^{1,2}(\Omega, |y|^\beta) \cap C_{c}(\Omega)$ follows by a standard density argument. 
\end{proof}
 
The Caccioppoli inequality is the first step to proving convergence in a Sobolev sense. It will also be useful when we remove the {\it a priori} dependence of our results on the Sobolev norm of the minimizer. 

\begin{lemma}[Caccioppoli Inequality]
\label{cacciop}
Let $B \subset \R^{n+1}$ be a ball of radius $r$ centered on $\R^n\times\{0\}$, and let $u \in W^{1,2}(B, |y|^\beta)$ be such that $\nabla\cdot (|y|^\beta \nabla u)=0$ weakly in $B\cap \{u>0\}$. Then 
$$\int_{\frac12 B} |y|^\beta |\nabla u|^2 \leq \frac{4}{r^2} \int_{B\setminus \frac12 B} |y|^\beta u^2.$$
\end{lemma}
\begin{proof}
Let $\eta$ be a Lipschitz function such that $\chi_{\frac12 B}\leq \eta\leq \chi_B$ and with $|\nabla \eta|\leq \frac{1}{r}$. By Lemma \ref{lemRadonMeasure}
$$0= \int_B u \eta^2 d\lambda = \int_B |y|^\beta \nabla u\cdot \nabla(u \eta^2).$$
By the Leibniz rule 
$$\int_B |y|^\beta  \eta^2 |\nabla u|^2 =- \int_B |y|^\beta  2 u \eta \nabla u\cdot \nabla\eta,$$ 
and using H\"older's inequality we get
$$\int_{\frac12 B} |y|^\beta |\nabla u|^2 \leq \int_B |y|^\beta \eta^2 |\nabla u|^2  \leq \int_B |y|^\beta  4 u^2 |\nabla\eta|^2 \leq \frac{4}{r^2} \int_{B\setminus \frac12 B} |y|^\beta u^2.$$
\end{proof}


\begin{lemma}\label{lemComparisons}
Let $u \in \HH^\beta(B_r)$ be a minimizer of $\mathcal{J}$ in $B_{2r}$ and $0\in F(u)$. Then 
$$ r^{-n/2}\norm{  \nabla u}_{L^2(\frac12 B_r;|y|^\Beta)} \leq  r^{-\alpha} \norm{u}_{L^\infty (B_r)} \leq \norm{u}_{\dot C^\alpha(B_r)} { \leq C \left(1+r^{-n/2} \norm{  \nabla u}_{L^2( B_{2r};|y|^\Beta)} \right)}.$$
\end{lemma}

\begin{proof}
The first inequality is Caccioppoli, the middle estimate is trivial and the last follows from \emph{P1}  in Theorem \ref{theoCRS}.
\end{proof}

\subsection{Compactness}

In the following lemma we prove the compactness of minimizers in the relevant Sobolev spaces. For convenience, we also detail several compactness results which were either already proven in \cite{CaffarelliRoquejoffreSire} or are standard consequences of the non-degeneracy estimates in Theorem \ref{theoCRS}. Nevertheless, we include full proofs here for the sake of completeness. We note here (as we did above and will do again below) that while we currently need to assume the uniform bound on the H\"older norm of the functions $u_k$ we can get rid of  this assumption in the light of the results of Section \ref{s:UniformBounds}.

\begin{lemma}[Compactness results]\label{lemCompactness}
 Let $\{u_k\}_{k=1}^\infty\subset \HH^\beta_{\loc}(\Omega)$ be a sequence of minimizers in a domain $\Omega\subset \R^{n+1}$ with $\norm{u_k}_{\dot C^\alpha(\Omega)}\leq E_0$ with non-empty free boundary. Then there exists a subsequence converging 
to some $u_0\in\HH^\beta_{\loc}(\Omega)$ such that for every bounded open set 
$G\subset \subset \Omega$ we have
 \begin{enumerate}
  \item\label{eka} $u_k\rightarrow u_0$ in $C^{\beta}(G)$ for 
every $\beta<\alpha$,
  \item $u_k\rightarrow u_0$ in $L^p(G)$ for every $p\leq\infty$,
  \item $\partial \{u_k>0\}\cap \bar G \rightarrow \partial\{u_0>0\}\cap \bar 
G$  in the
Hausdorff distance,
  \item $\chi_{\{u_k>0\}}\rightarrow\chi_{\{u_0>0\}}$ in $L^1(G)$, and
  \item $\nabla u_k \to  \nabla u_0$ in  $L^p(G;|y|^\beta)$ for every $p\leq 2$.
  \end{enumerate}
  \end{lemma}

\begin{proof}
 The first claim follows from uniform H\"older continuity and compact 
embeddings  of H\"older spaces. The claim (2) follows from (1) easily. 
 
 We now prove the third claim. Let $\epsilon>0.$  We will first  show that 
for $x \in 
\R^n$ we have 
\begin{equation}
\label{hausconv1}
d(x, 
F(u_0))>\epsilon \Rightarrow  d(x,F(u_k))>\frac\epsilon2
\end{equation}
for large $k.$ 
This implies that $F(u_k)\subset \{x\colon\, d(F(u_0),x)<2\epsilon\}$ for $k$ 
large enough.

Let $B(x,\epsilon)\subset F(u_0)^c.$ If $u_0$ is positive in $B(x,\epsilon)$ then 
it is bounded from below by a positive number in  $B(x,\epsilon/2).$ In this 
case $u_k$ are also positive in  $B(x,\epsilon/2)$ for large $k$ due to uniform 
convergence in $G$. Thus $B(x,\epsilon/2)\subset F(u_k)^c$ for large $k.$  If 
$u\equiv0$ in $B'(x,\epsilon)$ then due to the uniform convergence we know that  
for $k$ large enough $u_k<C\epsilon^\alpha$ in $B'(x,\epsilon)$, where $C$ is a constant given by \emph{P2} in Theorem \ref{theoCRS} 
so that   $u_k$ has no free boundary points in  $B(x,\epsilon/2)$ for all large $k.$  This proves \eqref{hausconv1}.

Next we will show that for all large $k$
\begin{equation}
\label{hausconv2}
F(u_0)\subset \{x\colon\, d(F(u_k),x)<\epsilon\}.
\end{equation}
 If this was not true we could find a point $x\in F(u_0)$ and a subsequence  of 
$u_k$ such that $B'(x,\epsilon)\subset F(u_k)^c$ for every $k$ in the subsequence.
If the subsequence contains infinitely many $u_k$ such that $u_k\equiv0$ in 
$B(x,\epsilon)$ then also $u_0\equiv 0$ due to uniform convergence.
Otherwise, the sequence contains infinitely many $u_k$ for which $B(x,\epsilon)$
is contained in the positive phase. In this case the non-degeneracy implies 
that 
in $B(x,\epsilon/2)$ we have $u_k>C\epsilon^\alpha,$ with $C$ independent of $k.$ Again 
uniform 
convergence implies the same lower bound for $u_0,$ which contradicts our 
choice 
 $x\in F(u_0).$  
 
 To show the fourth claim we notice that $F(u_0)$ has zero Lebesgue measure by the Lebesgue differentiation Theorem and the positive density of the zero phase. Take 
an 
open set $V\supset F(u_0)$ with $m(V\cap  G )<\epsilon.$ For large $k$ 
we have $F(u_k)\cup F(u_0)\subset  V\cap  G$, so  $\norm{\chi_{\{u_k>0\}}- \chi_{\{u_0>0\}}}_{L^1(G)}<\epsilon.$

Also the sequence is uniformly bounded in $H^{1,p}(G;|y|^\beta)$ by the Caccioppoli inequality. This implies 
by compactness \cite[1.31 Theorem]{HeinonenKilpelainenMartio} the  
weak convergence of $\nabla u_k$ in $L^p(G;|y|^\beta).$ To obtain strong convergence, use Lemma \ref{lemStronger} below.
 
\end{proof}

It remains to show that weak convergence implies strong convergence. 

\begin{lemma}\label{lemStronger}
Any sequence of minimizers  $\{u_k\}_{k=0}^\infty$ in $\Omega\subset \R^{n+1}$ with $u_k\rightarrow u_0$ uniformly and 
$\nabla u_k\rightharpoonup \nabla u_0$ weakly in  $L^2_{\loc}(\Omega,|y|^\beta)$ satisfies that $\nabla u_k\to \nabla u_0$ in  $L^2_{\loc}(\Omega,|y|^\beta)$.
\end{lemma}
\begin{proof}
Let $\eta\in C^{0,1}_c(\Omega)$ be a non-negative function. We claim that for every $\varepsilon>0$ there exists $j_0$ so that 
$$\int |y|^\beta \eta |\nabla u-\nabla u_j|^2\leq \varepsilon$$
for $j\geq j_0$.

First we isolate the main difficulty
$$\int |y|^\beta \eta |\nabla u_0-\nabla u_j|^2=\int |y|^\beta \eta (\nabla u_0-\nabla u_j)\cdot\nabla u_0  - \int |y|^\beta \eta (\nabla u_0-\nabla u_j)\cdot\nabla u_j.$$
By weak convergence, 
$$\left|\int |y|^\beta \eta (\nabla u_0-\nabla u_j)\cdot\nabla u_0\right|\leq \varepsilon/4 $$
for $j$ big enough. Note that this is true even if the $u_j$ are not minimizers. The bound on the second term, however, needs the minimization property. 

We observe that \begin{equation}\label{e:usedivergence} \int |y|^\beta \eta (\nabla u_0-\nabla u_j)\cdot\nabla u_j = \underbrace{\int |y|^\beta  (\nabla u_0-\nabla u_j)\cdot\nabla (\eta u_j)}_{=:I} - \underbrace{ \int |y|^\beta  u_j (\nabla u_0-\nabla u_j)\cdot \nabla \eta}_{=:II}.\end{equation}

To estimate $I$ in \eqref{e:usedivergence}, let $\lambda_j$ be the measures corresponding to $u_j$ from Lemma \ref{lemRadonMeasure}. By \rf{eqMeasureDef} we get that
$$\int |y|^\beta  (\nabla u_0-\nabla u_j)\cdot\nabla (\eta u_j) = \int \eta u_j \, d\lambda_0 - \int \eta u_j\, d\lambda_j.$$
Since $\lambda_j$ is supported on $\{u_j = 0\}$ we have that 
$$\int \eta u_j \, d\lambda_j=0$$
for every $j$ (including $j = 0$ as $u_0$ is also a minimizer to $\mathcal J$, see Corollary 3.4 in \cite{CaffarelliRoquejoffreSire}). 

 To finish the estimate on $I$ in \eqref{e:usedivergence} we observe that
$$\int \eta u_j \, d\lambda_0=\int \eta (u_j-u_0) \, d\lambda_0\leq \sup_{\supp\;  \eta} |u_j-u_0| \int \eta \, d\lambda_0 .$$
By uniform convergence on compact subsets, for $j$ big enough, $\sup_{\supp\;  \eta} |u_j-u_0|\leq \frac{\epsilon}{4\norm{\eta}_{L^1(\lambda_0)}}$. 

We turn towards estimating $II$ in \eqref{e:usedivergence}:
\begin{align}\label{eqbreak3}
|II|= \left| \int |y|^\beta  u_j (\nabla u_0-\nabla u_j)\cdot \nabla \eta\right|
\nonumber 	&  \leq  \left| \int |y|^\beta   (\nabla u_0-\nabla u_j)\cdot (u_0 \nabla \eta)\right| \\
	& \quad + \sup_{\supp\; \eta} |u_j-u_0|  \norm{\nabla u_0-\nabla u_j}_{L^2(\Omega,|y|^\beta)}\norm{\nabla \eta}_{L^2(\Omega,|y|^\beta)}.
 \end{align}
The first term goes to zero by weak convergence of $\nabla u_j$ to $\nabla u_0$. The second term satisfies 
$$\sup_{\supp \eta} |u_j-u_0|  \norm{\nabla u_0-\nabla u_j}_{L^2(\supp \eta,|y|^\beta)}\norm{\nabla \eta}_{L^2(\Omega, |y|^\beta)}\leq \varepsilon/4$$
for $j$ big enough, by uniform convergence and the uniform bound of $\norm{\nabla u_j}_{L^2(\supp\; \eta,|y|^\beta)}$ derived from the Caccioppoli inequality in Lemma \ref{cacciop} together with uniform convergence. 
\end{proof}

%
%
%
%
%
%
%
%
%
%
%
%

Lemma \ref{lemCompactness} implies that minimizers converge to minimizers (which was observed in Corollary 3.4 in \cite{CaffarelliRoquejoffreSire}), but also implies the stronger fact that the energy is continuous under this convergence:

\begin{corollary}\label{c:ContinuityofEnergy}
Let $u_k$ be a sequence of minimizers in $\Omega \subset \mathbb R^{n+1}$ with $u_k \rightarrow u_0$ locally uniformly and $\sup_k \|u_k\|_{\HH^\beta} < \infty$. Then $u_0$ is also a minimizer to $\mathcal J$ in $\Omega$ and for any $B \subset \subset \Omega$ we have $\mathcal{J}(u_k, B) \rightarrow \mathcal{J}(u_0, B)$. 
\end{corollary}

\section{Monotonicity formula and some immediate consequences}\label{s:monotonicity}

From \cite{Allen} we have the following monotonicity formula:
\begin{theorem}[Monotonicity formula, see { \cite[Theorem 4.3]{Allen}} ]\label{theoMonotonicityMinimum}
Let $u\in \HH^\beta(B_\delta(x_0))$ be a minimizer in $B_\delta(x_0)$ for the functional $\mathcal{J}$ with $x_0\in F(u)$.
Then the function
$$r\mapsto \Psi^u_r(x):=\Psi(r)=\frac{\mathcal{J}(u,B_r(x_0))}{r^{n}}  - \frac{\alpha}{ r^{n+1}}\int_{\partial B_r(x_0)} |y|^\Beta u^2 \, d\mathcal{H}^{n}$$
is defined and nondecreasing in $(0,\delta)$, and for $0<\rho<\sigma<\delta$, it satisfies
$$\Psi(\sigma)-\Psi(\rho) = \int_{B_\sigma(x_0)\setminus B_\rho(x_0)} |y|^\Beta\frac{2\left|\alpha u (x) -(x-x_0)\cdot \nabla u(x)\right|^2 }{|x_0-x|^{n+2}}dx \geq 0.$$
\end{theorem}

As a consequence, the blow-up limits are cones, in the sense of the following corollary.
\begin{corollary}\label{coroWeakLimits}
Let $u\in \HH^\beta(B_\delta(x_0))$ be a  minimizer in $B_\delta(x_0)$ with $x_0=(x_0',0)$. Consider a decreasing sequence $0<\rho_k\xrightarrow{k\to\infty}0$ and the associated rescalings $u_k(x):=\frac{u(x_0+\rho_k x)}{r^\alpha}$.
Then  the \emph{Allen-Weiss density}
$$\Psi^{u}_0(x_0):=\lim_{r\searrow 0} \Psi^{u}_r(x_0)$$
is well defined. Furthermore, for every bounded open set $D\subset\R^{n+1}$ and $k\geq k(D)$ this subsequence $u_{k}$ is bounded in $H^{1,2}(D;|y|^\beta)$ and, passing to a subsequence $u_{k_j}$, converges (in the sense of Lemma \ref{lemCompactness}) to $u_0$ which is a globally defined minimizer of $\mathcal J$ that is homogeneous of degree $\alpha$.
\end{corollary}
The proof is the same as  in \cite[Theorem 2.8]{WeissMinimum}

\begin{remark}[Non-uniqueness of blow-ups]
We call the function $u_0$ appearing in Corollary \ref{coroWeakLimits} a \emph{blow-up} of $u$ at $x_0$. A priori, the function $u_0$ may depend on the subsequence $u_{k_j}$. However, a simple scaling argument shows that for all radii $r \geq 0$ and all blow-ups $u_0$ to $u$ at $x_0$ we have $$\Psi^{u_0}_r(0) \equiv \Psi^u_0(x_0).$$ 
\end{remark}

\subsection{Dimension reduction}\label{secReduction}
We use the homogeneity of the blow-ups to obtain dimension estimates on the points in the free boundary for which there exists a non-flat blow-up. This process is known as ``dimension reduction" and has been applied to a variety of situations (see \cite{WeissMinimum} for its application to the Bernoulli problem).

The first lemma shows that blow-up limits of blow-up limits have additional symmetry:

\begin{lemma}\label{lemConeBlowup}
Let $u\in \HH^\beta_{\loc}(\R^{n+1})$ be an $\alpha$-homogeneous  minimizer of $\mathcal{J}$ and let $x_0\in F(u)\setminus\{0\}$. Then any blow-up limit $u_0$ at $x_0$ is invariant in the direction of $x_0$, i.e., for every $x\in \R^{n+1}$ and every $\lambda\in\R$,
$$u_0(x+\lambda x_0)=u_0(x) .$$
\end{lemma}

\begin{proof}
Let $x \in \R^{n+1}$, and consider its decomposition $x=\widetilde{x}+\lambda x_0$ with $\widetilde{x}\in \langle x_0 \rangle^\bot$. We only need to check that \begin{equation}\label{eqTargetIdentity}
u_0(x)=u_0(\widetilde{x}).
\end{equation}

\begin{figure}[ht]
\center
\begin{tikzpicture}[line cap=round,line join=round,>=triangle 45,x=0.7cm,y=0.7cm]
\clip(-3.96,-2.14) rectangle (6.78,3.7);
\draw [line width=.5pt, ] (4.266094252316256,2.487656762197886)-- (-2.78,0.58);
\draw [line width=.5pt] (3.3643812272510014,2.2435273355638077)-- (-2.78,0.58);
\draw [line width=.5pt] (-2.78,0.58)-- (3.38,0.64);
\draw [line width=.5pt] (3.38,0.64)-- (3.3620891085978095,2.4788515172915373);
\draw [line width=.5pt] (3.3620891085978095,2.4788515172915373)-- (4.266094252316256,2.487656762197886);
\draw [line width=.5pt] (3.3643812272510014,2.2435273355638077)-- (3.38,0.64);
\draw [line width=.5pt] (3.38,0.64) circle (2.420080008203852);
\draw (4.2,3.62) node[anchor=north west] {$B_{\rho_k}(x_0)$};
\draw (2.6,3.1) node[anchor=north west] {$P_1$};
\draw (4.2,2.72) node[anchor=north west] {$P_2$};
\draw (3.2,2.4) node[anchor=north west] {$P_3$};
\draw (-3.22,0.5) node[anchor=north west] {$O$};
\draw (3.38,0.88) node[anchor=north west] {$x_0$};
\begin{scriptsize}
\draw [fill=blau] (-2.78,0.58) circle (1.pt);
\draw [fill=blau] (3.38,0.64) circle (1.pt);
\draw [fill=blau] (3.3620891085978095,2.4788515172915373) circle (1.pt);
\draw [fill=blau] (4.266094252316256,2.487656762197886) circle (1.pt);
\draw [fill=blau] (3.3643812272510014,2.2435273355638077) circle (1.pt);
\end{scriptsize}
\end{tikzpicture}
\caption{The distance $\dist(P_1,P_3)=\mathcal{O}(\rho_k^2)$.}\label{figTrigonometry}
\end{figure}
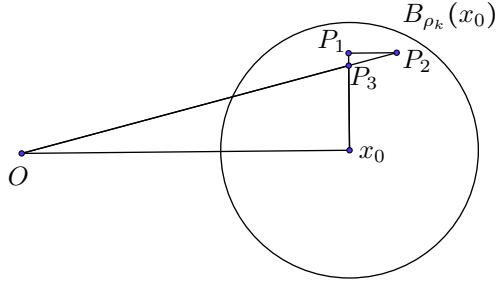

Consider a ball $B=B(0,r) \subset \R^{n+1}$ so that $\widetilde{x}, x\in B$.
Let $\{\rho_k\}$ be a sequence of radii converging to zero and such that $u_k(x):=\frac{u(x_0+\rho_k x)}{\rho_k^\alpha}$ converges  to $u_0$
uniformly on $B_r$. For $k$ big enough, $\norm{u_k-u_0}_{L^\infty(B_r)}<\varepsilon$. Then,
\begin{equation}\label{eqPartialStep}
|u_0(x)-u_0(\widetilde{x})|\leq 2\varepsilon + |u_k(x)-u_k(\widetilde{x})|.
\end{equation}

To control the last term above, we use the homogeneity of $u$.
Writing $P_1:=x_0+\rho_k \widetilde{x}$ and $P_2:=x_0+\rho_k x$ we have $\rho_k^\alpha u_k(\widetilde{x})=u(P_1)$ and $\rho_k^\alpha u_k(x)=u(P_2)$. Let $P_3$ be the intersection between the line through $P_1$ and $x_0$ and the line through the origin and $P_2$ (see Figure \ref{figTrigonometry}). By homogeneity of $u$
$$u(P_2)=u(P_3)\left(\frac{|P_2|}{|P_3|}\right)^\alpha=u(P_3)\left(1 \pm \frac{|P_2-P_3|}{|P_3|}\right)^\alpha.$$
Thus,
\begin{align*}
\rho_k^\alpha|u_k(x)-u_k(\widetilde{x})|
    & \leq \left|u(P_1) - u(P_3)\left(1 + \mathcal{O}(\rho_k)\right)^\alpha \right|
           \leq \left|u(P_1) - u(P_3)\right| + \left|u(P_3)\right|\mathcal{O}(\rho_k).
\end{align*}
By Thales' Theorem, $|P_1- P_3|=\frac{|P_1- P_2||P_3-x_0|}{|x_0|}=\mathcal{O}(\rho_k^2)$ and using the $\dot{C}^\alpha$ character of $u$ and the fact that $u(x_0)=0$, we get
\begin{align*}
\rho_k^\alpha|u_k(x)-u_k(\widetilde{x})|
     &    \leq \norm{u}_{\dot{C}^\alpha}\left( \left|P_1- P_3\right|^\alpha  + \left|P_3\right|^\alpha \mathcal{O}(\rho_k)\right)
         =  \mathcal{O}(\rho_k^{2\alpha} )+ \mathcal{O}(\rho_k),
\end{align*}
and \rf{eqTargetIdentity} follows by \rf{eqPartialStep} since $\rho_k\to 0$.
\end{proof}

We then recall that a minimizer with a translational symmetry is actually a minimizer without that symmetry in one dimension less. This is known as ``cone splitting":

\begin{lemma}\label{lemDimensionReduction}
Let $u\in \HH^\beta_{\loc}(\R^{n+1})$ be an $\alpha$-homogeneous minimizer of $\mathcal{J}$ in $\R^{n+1}$ which is invariant in the direction $e_n$. Then  $\widetilde{u}(x',y):=u(x',0,y)$ is a minimizer of $\mathcal{J}$ in one dimension less.
 \end{lemma}
 \begin{proof}
 The proof is a slight variation of {\cite[Proof of Lemma 3.2]{WeissMinimum}}.
  \end{proof}

Next we provide a non-standard proof of Proposition \ref{propoTrivial}, that is, to show that the trivial solution is a minimizer. We use \emph{P\ref{itemBlowup}} in a sequence of conveniently chosen blow-ups and a dimension reduction argument, based on the following lemma. Note that the proposition could also be proven via a classical dimension reduction argument.

\begin{proof}[Proof of Proposition \ref{propoTrivial}]
Consider a non-zero minimizer $u$ with non-empty free boundary (see \cite[Pro\-position 3.2]{CaffarelliRoquejoffreSire} for its existence), choose a free boundary point $x_0\in F(u)$ and consider $u_0$ to be a blow-up weak limit at this point, which exists and is $\alpha$-homogeneous by Lemma \ref{coroWeakLimits}. Then $u_0$  is also  a global minimizer by \emph{P\ref{itemBlowup}} and not null by the nondegeneracy condition.

Next we argue by induction: given $0\leq j\leq n-2$ let $u_j$ be an $\alpha$-homogeneous global minimizer different from $0$ such that it is invariant in a $j$-dimensional linear subspace $H_j\subset \R^n$, i.e., for every $v\in H_j$ and every $x'\in\R^{n}$,
$$u_j(x',y)=u(x'+v,y).$$

Consider a point $x_j \in F(u_j)\setminus (H_j\times\{0\})$ which exists as long as $j<n-1$ by the interior corkscrew condition and positive density, and let $u_{j+1}$ be a blow-up limit at this point, which is again an $\alpha$-homogeneous global minimizer. We claim that $u_{j+1}$ is invariant in fact in the $(j+1)$-dimensional subspace $H_j+\langle x_j' \rangle$.

Indeed $u_{j+1}$ is invariant in $\langle x_j'\rangle$ by Lemma \ref{lemConeBlowup}. On the other hand, since $u_j$ is invariant in $H_j$, so are the functions in the blow-up sequence and, thus, $u_{j+1}$ is invariant in $H_j$. Thus, for $v\in H_j$, $v_0\in \langle x_j'\rangle$ and $x\in\R^{n+1}$ we get
$$u(x+v+x_j')=u(x+v)=u(x),$$
and the claim follows.

Thus, after $n-1$ steps, we obtain $u_{n-1}$ which is an $\alpha$-homogeneous global minimizer invariant in an $(n-1)$-dimensional space $H_{n-1}$, with non-empty free boundary. Thus,
$$u_{n-1}(x',0)=C_{n,\alpha}(x'_n)_+^\alpha,$$
where the constant is given by \emph{P\ref{itemNormal}}.
The proposition follows by  Proposition \ref{propoExtension}.
\end{proof}

\subsection{Upper semicontinuity}
Next we show that Allen-Weiss' energy at a fixed radius is continuous both with respect to the minimizer and with respect to the point:
\begin{lemma}\label{lemContinuity}
Let $u_j\in \HH^\beta_{\loc}(\Omega)$ be minimizers of $\mathcal{J}$ in $\Omega$ and $u_j \to u_0$ in the sense of Lemma \ref{lemCompactness}. Then, for $x_j\to x_0$ and $r<\dist(x_0,\partial\Omega)$, 
$$  \Psi_r^{u_j}(x_j) \xrightarrow{j\to\infty} \Psi_r^{u_0}(x_0).$$
\end{lemma}
\begin{proof}
Let $\varepsilon>0$. We want to check that  for $j$ big enough, 
$$\left|  \Psi_r^{u_j}(x_j) -\Psi_r^{u_0}(x_0)\right|\leq \varepsilon.$$

We will consider the three terms of the energy separately. 
For the first term, 
$$ \int_{B_r(x_j)}|y|^\Beta |\nabla u_j|^2- \int_{B_r(x_0)}|y|^\Beta |\nabla u_0|^2 \leq r^n \varepsilon/3$$
follows from the $L^2$ convergence of the gradients. Indeed, if $\delta_j:=|x_j-x_0|\leq \delta $ for $j$ big enough and $B_{r+\delta}\subset \Omega$, then
\begin{align*}
 \int_{B_r(x_j)}|y|^\Beta |\nabla u_j|^2- \int_{B_r(x_0)}|y|^\Beta |\nabla u_0|^2 
 	& \leq
 \int_{B_r(x_j)}|y|^\Beta \left(|\nabla u_j|^2- |\nabla u_0|^2\right)+ \int_{B_r(x_j)\Delta B_r(x_0)}|y|^\Beta |\nabla u_0|^2  \\
 	& \leq \int_{B_{r+\delta}(x_j)}|y|^\Beta \left(|\nabla u_j|^2- |\nabla u_0|^2\right)+ \int_{(B_{r+\delta_j}\setminus B_{r-\delta_j})(x_0)}|y|^\Beta |\nabla u_0|^2  \\
	&\leq r^n \varepsilon/3.
	\end{align*}

For the measure, we estimate
$$\left|\int_{B_r(x_j)'} \chi_{\Omega_+(u_j)} dm - \int_{B_r(x_0)'} \chi_{\Omega_+(u_0)}dm \right| \leq r^n \varepsilon/3$$
for $j$ big enough as a consequence of $\chi_{\Omega_+(u_j)}\to \chi_{\Omega_+(u_0)}$ in $L^1_{\loc}$ as before.
The fact that
$$\alpha \left| \int_{\partial B_r(x_j)} u_j^2- \int_{\partial B_r(x_0)} u_0^2  \right| \leq r^{n+1} \varepsilon/3 $$
for $j$ big enough is a straight consequence of the uniform convergence and the continuity of $u_0$.
\end{proof}

It is well known that the limit of a decreasing sequence of continuous functions is upper semicontinuous (see {\cite[Theorem 1.8]{DalMaso}}). The monotonicity formula also implies the following result.
\begin{lemma}\label{lemUSC}
Let $u_j\in \HH^\beta_{\loc}(\Omega)$  be minimizers of $\mathcal{J}$ in $\Omega$ and $u_j \xrightarrow{j\to\infty}u_0$ in the sense of Lemma \ref{lemCompactness}, with $x_j\in F(u_j)$ for $j\in \N$. Then, if $x_j\to x_0$ and $r_j\to 0$,
$$\limsup_{j} \Psi_0^{u_j}(x_j)\leq \limsup_{j} \Psi_{r_j}^{u_j}(x_j) \leq \Psi_0^u(x_0).$$
\end{lemma}
\begin{proof}
The first inequality comes from monotonicity.

To see that
$$\limsup_{j} \Psi_{r_j}^{u_j}(x_j) \leq \Psi_0^{u_0}(x_0),$$
it is enough to check that for every $r>0$
$$ \limsup_{j} \Psi_{r_j}^{u_j}(x_j) \leq \Psi_r^{u_0}(x_0),$$
or using monotonicity it suffices to show that for every $\varepsilon>0$ and $j$ big enough, 
$$  \Psi_r^{u_j}(x_j) - \Psi_r^{u_0}(x_0) \leq \varepsilon.$$
But this is true for $j$ big enough because the left-hand side converges to $0$ by the continuity of the energy from Lemma \ref{lemContinuity}.
\end{proof}

\section{Measure-theoretic properties}\label{s:finperimeter}

\subsection{Finite perimeter}
We will show that $\Omega'_+(u)$ is a set of locally finite perimeter. Then  $F_{\rm red}(u)$  will coincide with the measure-theoretic reduced boundary by the $\epsilon$-regularity theorem, see \cite[Sections 4.6 and 4.7]{AltCaffarelli}.

\begin{definition}
For every $0<\alpha<1$ we can define
 $$k^*_\alpha:=\inf \left\{k\in\N: \exists\mbox{ an  $\alpha$-homogeneous  minimizer $u\in \HH^\beta_{\loc}(\R^{k+1})$ such that $\Sigma(u)=\{0\}$}\right\}.$$
\end{definition}
Note that, to the best of our knowledge, there is no result showing that $k^*_\alpha$ needs to be finite.

\begin{lemma}\label{lemTrivialIsTheChosen}
Let $u$ be an $\alpha$-homogeneous minimizer of $\mathcal{J}$ in $\R^{n+1}$ with $n < k^*_\alpha$. Then $u$ is a rotation of the trivial solution.
\end{lemma}
See \cite[Section 3]{WeissMinimum} for the proof.

From the positive density properties, we know that $k^*_\alpha\geq 2$.  From the homogeneity of the blow-ups we find out that the free boundary in $\R^{1+1}$ is in fact  a collection of isolated points. Later in Theorem \ref{FullRegDim2} we will show that in fact $k^*_\alpha\geq 3.$
\begin{lemma}[Isolated singularities]
\label{IsolSing}
Let $u\in \HH^\beta_{\loc}(\Omega)$ for $\Omega\subset \R^{1+1}$ be a minimizer of $\mathcal{J}$ in $\Omega$. Then $F(u)$ has no accumulation points in $\Omega$.
\end{lemma}
\begin{proof}
Arguing by contradiction, we assume that $F(u)$ has an interior accumulation point which, without loss of generality, we assume to be the origin. 

Let $(x_k,0)$  be a sequence of singular points converging to $0$ with $x_k >0$. Consider the blow-up rescaling $u_k(x):=\frac{u(x_k x)}{x_k^\alpha}$. Note that $u_k(0,0)=u_k(1,0)=0$. Moreover, by the interior corkscrew condition, there exist $z_k\in (1/2, 3/2)$ such that $u_k|_{B'_c(z_k, 0)}>0$, so $u(z_k,0)\gtrsim C$ by the non-degeneracy condition. 

Choosing a subsequence, we may assume that $z_k\to z_0\geq 1/2$, and $u_k\to u_0$ in the sense of Lemma \ref{lemCompactness}. In particular $u_0$ is homogeneous by Corollary \ref{coroWeakLimits}, reaching a contradiction with the fact that $u_0(1,0)=0$ and $u_0(z_0,0)\gtrsim C$.
\end{proof}

We will prove the local finiteness of the perimeter of the free boundary adapting a proof of De Silva and Savin in \cite{deSilvaSavinJEMS}. Our proof is essentially the same, but we repeat it for the sake of completeness.

As in \cite{deSilvaSavinJEMS} we say that a set $A\subset \R^n$ satisfies the property (P) if following holds: for every $x\in A$ there exists an $r_x>0$ such that for every $0<r<r_x$, every subset $S$ of $B(x,r)\cap A$ can be covered with a finite number of balls $B(x_i,r_i)$ with $x_i\in S$ such that 
\begin{equation}
 \label{PropP}
 \sum_i r_i \leq r^\alpha/2.
\end{equation}

\begin{lemma}
\label{MinimDim}
 If  $\mathcal H^t (\Sigma(U))=0$ for some $\alpha>0$ and for every minimal cone $U$ in $\R^{n+1}$ then $\mathcal H^t (\Sigma(u))=0$ for every minimizer $u$ of $\mathcal J$ defined on $\Omega\subset \R^{n+1}$ 
\end{lemma}

\begin{proof}
 We first show that $\Sigma(u)$ satisfies the property (P). If (P) does not hold we find a point $y\in \Sigma(u)$ for (P) is violated for a sequence $r_k\rightarrow 0.$ We consider the blow-up sequence 
 \begin{equation}
  u_{r_k}(x)=r_k^{-\alpha}u(y+r_k x).
 \end{equation}
By Corollary \ref{c:ContinuityofEnergy} we may assume, by taking  a subsequence,  that  $u_{r_k}$ converges to a minimal cone $U$. By our assumptions we may cover $\Sigma(U)\cap B(0,1)$ with a finite collection of balls $\{B(x_i,\frac{\rho_i}{10})\}_{i=1}^k$  with 
$$
\sum_i \rho_i^t\leq \frac12.
$$

By Lemma \ref{lemCompactness} we know that free boundaries converge in Hausdorff sense and thus the set
$F(u_{r_k})\cap B(0,1)\setminus \bigcup_i B(x_i,\rho_i/5)$ is flat for all large $k.$ From Theorem \ref{th:iof} we infer that all singularities must be covered by the same balls, that is, for all $k\geq k_0$
\begin{equation}
 \Sigma(u_{r_k})\cap B(0,1)\subset \bigcup_i B(x_i,\rho_i/5).
\end{equation}
After rescaling we see that $u$ satisfies the condition for property (P) in the ball $B(y,r_k),$ which is a contradiction. Therefore the property (P) holds as claimed.

Consider the set $D_k:=\{y\in \Sigma(u): r_y\geq1/k \}.$ Fix a point $y_0\in D_k.$  By property (P) applied to $r_0=1/k$ we find  a finite cover of $D_k\cap B(y_0,r_0)$ with balls $B(y_i,r_i),$ $y_i\in D_k,$ satisfying
$$
\sum_i r_i^t\leq r_0^t/2.
$$
Similarly, for each ball $B(y_i,r_i)$ in the cover we use the property (P) to find a finite number of balls $B(y_{ij}, r_{ij}),$ $y_{ij}\in D_k,$ which cover $D_k\cap B(y_i,r_i)$ and satisfy
$$
\sum_j r_{ij}^t\leq r_i^t/2, 
$$
and thus $\sum_{i,j} r_{ij}\leq r_0^t/4$.
By repeating the argument $N$ times we obtain a cover of $D_k\cap B(y_0,r_0)$ by balls $B(z_l,r_l)$ which satisfies
$$
\sum_l r_l^t \leq 2^{-N}r_0^t.
$$
This implies that $\mathcal H^t(B(y_0,r_0)\cap D_k)=0$ and thus $\mathcal H^t(D_k)=0.$ By countable additivity we obtain the claim.

\end{proof}

\begin{lemma}
\label{MeasReduction}
 If  $\mathcal H^t (\Sigma(U))=0$ for some $t>0$ and for every minimal cone in $\R^{n+1}$, we then have that $\mathcal H^{t+1} (\Sigma(V))=0$ for every minimal cone $V$ in $\R^{(n+1)+1}.$ 
\end{lemma}
\begin{proof}
Without loss of generality we may assume $\Sigma(V)\neq \{0\}.$ 
Let $x\in \Sigma(V)\setminus \{0\}.$
By Corollary \ref{c:ContinuityofEnergy} the blow-ups at any point of $\Sigma(V)\cap \partial B$ converge to a minimal cone  in dimension $(n+1)+1$ up to a subsequence. Let $V_x$ be a blow-up at $x.$ By Lemma \ref{lemDimensionReduction} $V_x$ is a minimal cone which is invariant in at least one direction. By Lemma \ref{lemDimensionReduction}, using our assumption this implies that $\mathcal H^{t+1}(\Sigma(V_x))=0$,  and thus the singular set of every possible blow-up cone of any minimizer $V$  has zero $\mathcal H^{t+1}$-measure. 

Arguing as in Lemma \ref{MinimDim} we obtain  $\mathcal H^{t+1}(\Sigma(V))=0.$

%
%
%
%
%

\end{proof}

Combining Lemmas \ref{IsolSing}, \ref{MinimDim} and  \ref{MeasReduction} we obtain the following corollary. Notice that we will be able to replace $n-1$ by $n-2$ by Theorem \ref{FullRegDim2}.

\begin{corollary}\label{coroZeroHn}
Every minimizer satisfies 
$$\mathcal{H}^{n-1}(\Sigma(u))=0.$$
\end{corollary}

\begin{lemma}
\label{l:cover}
 Let $u\in \HH^\beta (2B)$ be a minimizer of $\mathcal J$ in $2B$ with $\norm{u}_{\dot C^\alpha(2B)}<E_0.$ Then there exists a constant $C$ depending on $n$, $\alpha$ and $E_0$  and a finite collection of balls $\{B(X_i,r_i)\}$ s.t.
 \begin{equation}\label{eq:noballs}
  \mathcal H^{n-1}\left ( (F(u)\cap B)\setminus \bigcap_{i= 1}^m B(X_i,r_i) \right)\leq C
 \end{equation}
and
 \begin{equation}
 \label{eq:balls}
  \sum_{i= 1}^m r_i^{n-1} \leq \frac12.
 \end{equation}
\end{lemma}
\begin{proof}
 Proof is by contradiction. For $k\in\N$ assume we have $\norm{u_k}_{\dot C^\alpha(2B)}< E_0$ and the left-hand side of \eqref{eq:noballs} is bounded below by $k>0$ for every collection of balls satisfying \eqref{eq:balls}. By Lemma \ref{cacciop} we know the sequence $ u_k$ is bounded in $\HH^\beta(B).$  
  Taking  a subsequence we may assume that $u_k$ converges locally uniformly to a minimizer  $u$ (see Corollary \ref{c:ContinuityofEnergy}).
 
 By Corollary \ref{coroZeroHn} the set of singularities $\Sigma(u)$ has $\mathcal H^{n-1}$ -measure zero and thus they can be covered with finitely many balls $B_i$  satisfying \eqref{eq:balls}. 
 
 Since $F(u)\setminus \Sigma(u)$ is a $C^{1,\gamma}$-surface by Theorem \ref{th:iof}, using the Hausdorff convergence of the free boundaries we apply again Theorem \ref{th:iof} to see that $F(u_k)\cap B_1\setminus \bigcup_{i=1}^M B_i$ are also  $C^{1,\gamma}$-surfaces converging to $F(u)\cap B_1\setminus \bigcup_{i=1}^M B_i$ uniformly in the  $C^{1}$-norm. This is  a contradiction with the assumption that the Hausdorff measure blows up as $k$ goes to $\infty.$ 
\end{proof}

The fact that the free boundary has finite perimeter follows now from the same iteration argument as \cite[Lemma 5.10]{deSilvaSavinJEMS}.

\begin{lemma}\label{lemFinite}
 Let $u$ be as in Lemma \ref{l:cover}.
 Then for some constant $C$ depending only on $E_0$,
 \begin{equation}
  \mathcal H^{n-1}\left( F(u)\cap B\right)\leq C.
 \end{equation}

\end{lemma}
\begin{proof}
 By Lemma \ref{l:cover} we find a finite collection of balls $B_{r_i}$ such that
 \begin{equation}
  F(u)\cap B \subset \Gamma \cup  \bigcup B_{r_i},
 \end{equation}
with $\mathcal H^{n-1}(\Gamma)\leq C$ and $\sum r_i^{n-1}\leq \frac12.$

Applying Lemma \ref{l:cover} again for each ball $B_{r_i}$ we have
\begin{equation}
 F(u)\cap B_{r_i} \subset \Gamma_i \cup  \bigcup B_{r_{ij}},
\end{equation}
with $\mathcal H^{n-1}(\Gamma_i)\leq Cr_i^{n-1}$ and $\sum r_{ij}^{n-1}\leq \frac12r_i^{n-1}.$
Moreover, we have 
$$
\mathcal H^{n-1}\left((F(u)\cap B_1)\cap \bigcup_{i,j} B_{r_{ij}}\right)\leq\mathcal H^{n-1}(\Gamma)+\sum \mathcal H^{n-1}(\Gamma_i)\leq C \left(1+\sum_{i,j} r_{ij}^{n-1}\right)\leq C\left(1+\frac12\right).
$$

Continuing inductively, after $k$ steps we have  that
\begin{equation}
 F(u)\cap B_1 \subset \Gamma' \cup  \bigcup_{q=1}^N B_{r_{q}},
\end{equation}
with 
$$
\mathcal H^{n-1}(\Gamma')\leq C \left(\sum_{i=0}^k 2^{-i}\right)\leq 2C,
$$
and $\sum r_{q}^{n-1}\leq 2^{-k}.$
This gives the claim.
\end{proof}

Finally the fact that $\{u>0\}\cup \Omega$ has locally finite perimeter in $\Omega$ follows from the previous lemma and well-known results of Federer, see for example \cite[Prop. 3.62]{AFP} or \cite[4.5.11]{Federer}.

\subsection{Energy gap}
Next we will check that the Allen-Weiss density can also be used to identify singular points. First let us state a useful identity for minimizers (which is also valid in the context of variational solutions in the sense of \cite{WeissWeak}).
\begin{lemma}[{See \cite[Proposition 3.4]{Allen}}]\label{LemMagic}
Let $u\in \HH^\beta_{\loc}(\Omega)$ be a minimizer to \rf{eqLocalizedFunctional} in $\Omega$. For every $B\subset \subset\Omega$ we have
\begin{equation}\label{eqMagicFormula}
\int_{B}  |y|^\Beta |\nabla u|^2=\int_{\partial B} |y|^\Beta  u \nabla u \cdot \nu  \, d\mathcal{H}^{n}.
\end{equation}
\end{lemma}

Let $u$ be a minimizer and $x_0\in F(u)$. If we consider a blow-up $u_0$ at $x_0$, then 
$$\Psi^{u}_0(x_0)=\Psi^{u_0}_1(0)=\int_{B_1} |y|^\Beta |\nabla u_0|^2 + m(\{u_0>0\}\cap\R^n\cap B_1)  - \alpha \int_{\partial B_1} |y|^\Beta u_0^2 \, d\mathcal{H}^{n} .$$
By Lemma \ref{LemMagic} we get
$$\Psi^{u_0}_1(0) =\int_{\partial B_1(x_0)} |y|^\Beta  u_0 \nabla u_0 \cdot \nu  \, d\mathcal{H}^{n}+ m(\{u_0>0\}\cap\R^n\cap B_1)  - \alpha \int_{\partial B_1} |y|^\Beta u_0^2 \, d\mathcal{H}^{n} .$$
Since $\nabla u_0(x)  \cdot \nu (x) = \frac{\alpha }{|x|} u_0(x) $ almost everywhere on the sphere, the first and the third terms cancel out and we obtain
$$ \Psi^{u_0}_1(0)= m(\{u_0>0\}\cap B_1'). $$
Thus, the density $\Psi^{u}_0$ at a free boundary point is given by the area of the positive phase of any blow-up at the same point.

We write $\omega_n:=m( B_1')$ for the volume of the $n$-dimensional ball. 
\begin{proposition}\label{propoAlmostGap}
Every homogeneous minimizer $u\in \HH^\beta_{\loc}(\R^{n+1})$ has density 
$$\Psi^u_1(0)= m(\{u>0\}\cap B_1') \geq \frac{ \omega_n}2  ,$$ and equality is only attained when $u$ is the trivial minimizer. 
\end{proposition}
\begin{proof}
%
Let $u$ be a minimizer such that $\Psi^{u}_1(0)\leq \frac{ \omega_n}2$. 

Let $x_1\in F_{red}(u) $. Being a regular point, $\Psi^u_0(x_1)=\frac{ \omega_n}2$. On the other hand, by the homogeneity and the continuity in Lemma \ref{lemContinuity},
$$\lim _{r\to\infty}\Psi^u_r(x_1)=\lim _{r\to\infty}\Psi^u_1(x_1/r)=\Psi^u_1(0)\leq \frac{ \omega_n}2.$$
Combining both assertions with the monotonicity of $\Psi$ we get that $\Psi^u_r(x_1)\equiv \frac{ \omega_n}2$. But using the second formula in Theorem \ref{theoMonotonicityMinimum}, one can see that this is  true only whenever $\Psi$ is $\alpha$-homogeneous with respect to $x_1$. Thus, $u$ is $1$-symmetric and invariant in the direction of $\langle x_1\rangle$.

Since $\Omega_0$ is a finite perimeter set (see Section \ref{s:finperimeter}), $F_{red}(u)$ has full $\mathcal{H}^{n-1}$ measure on $F(u)$.  Thus, we can find $x_1,\dots x_{n-1} \in F_{red}(u) $ linearly independent. By the previous discussion $u$ is invariant on an $(n-1)$-dimensional affine manifold and, thus, it is the trivial solution.
\end{proof}

\begin{corollary}[Energy Gap]
There exists $\overline{\epsilon}>0$ depending only on $n$ and $\alpha$ such that every minimizer $u\in \HH^\beta_{\loc}(\Omega)$ and every singular point  $x_0\in \Sigma(u)$ satisfy  
$$ \Psi^{u}_1(x_0)-\frac{ \omega_n}2 \geq \overline{\epsilon}.$$
\end{corollary}
\begin{proof}
Assume the conclusion to be false. Then there exist $u_j$ minimizers in $B_1$ with 
$$ \Psi^{u_j}_1(0)\leq \frac{ \omega_n}2 + 1/j.$$ 
Passing to a subsequence, $u_j\to u_0$ as in Lemma \ref{lemCompactness}. Using Lemma \ref{lemUSC} we get that
$$ \Psi^{u_0}_1 (0)= \lim_j \Psi^{u_j}_1(0) \leq \frac{\omega_n}2 .$$
But then $u_0$ is the trivial cone by Proposition \ref{propoAlmostGap}. Since $F(u_j)\to F(u)$ in the Hausdorff distance, using $\epsilon$-regularity (see Theorem \ref{theoImprovement}) we get that $u_j$ is the trivial cone for $j$ big enough. 
\end{proof}
The value $\overline{\epsilon}$ above depends on the constants and on $\norm{u}_{\dot C^\alpha}$ in a neighborhood of $x_0$. In the next section we will show that $\overline{\epsilon}$ does not depend on $u$ at all.






\section{Full regularity in $\R^{2+1}$}\label{s:r2}

In the case of $n=2$, we prove full regularity of the free boundary for minimizers of our functional. Note that this result does not depend on the previous sections except that we use dimension reduction and blow-ups to deduce regularity of the free boundary.

\begin{theorem}
\label{FullRegDim2}
Let $n=2$. Then there is no singular minimal cone. In particular, the free boundary $F(u)$ of every minimizer $u$ is $C^{1,\alpha}$ everywhere.
\end{theorem}
\begin{proof}
We follow closely the arguments in \cite[Theorem 5.5]{deSilvaSavinJEMS}, building on \cite{SVcone}. The case $\beta=0$ has been considered in \cite{deSilvaSavinJEMS}. The idea is to construct a competitor by a perturbation argument. We note at this point that the argument is two dimensional in nature and does not generalize to higher dimensions. Recall the functional under consideration:
$$
\mathcal J (u,\Omega) = \int_{\Omega} |y|^\beta\,|\nabla u|^2 +m(\left \{ u>0\right \} \cap \mathbb R^n \cap \Omega).
$$
Let $V$ be a non trivial minimal cone. Define, as in \cite{deSilvaSavinJEMS}, the Lipschitz continuous function 
\begin{equation}
\psi_R(t)=
\left \{
\begin{array}{ccc}
1,\,\,\,\,\,\ 0 \leq t \leq R,\\
2-\frac{\ln(t)}{\ln(R)},\,\,\,\,\,\ R \leq t \leq R^2,\\
0,\,\,\,\,\,\  t \geq R^2
\end{array} \right .
\end{equation}
Define now the bi-Lipschitz change of coordinates
$$
Z(x',y)=(x',y)+\psi_R(|(x',y)|)e_1
$$
and set $V^+_R(Z)=V(x',y)$. Clearly, one has 
$$
D_{(x',y)} Z=\text{Id}+A
$$
where $\|A\| \leq |\psi'_R(|(x',y)|) | <<1$. Defining now $V^-_R$ exactly as $V^+_R$ changing $\psi_R$ into $-\psi_R$, the very same computation as in \cite{deSilvaSavinJEMS}  gives 

$$
\mathcal J(V^+_R, B_{R^2}) +\mathcal J(V^-_R, B_{R^2}) \leq 2 \mathcal J(V, B_{R^2})+\int_{B_{R^2}} |y|^\beta \, |\nabla V|^2 \|A\|^2.
$$
Now, we have 
$$
\int_{B_{R^2}} |y|^\beta \, |\nabla V|^2 \|A\|^2= \int_R^{R^2} \int_{\partial B_r} |y|^\beta \, |\nabla V|^2 \|A\|^2 \, d\mathcal{H}^n\,\,dr. 
$$
Now since $V$ is homogeneous of degree $\alpha$ by assumption, the function $g(x,y)=|y|^\beta  \, |\nabla V|^2$ is homogeneous of degree $\beta+2\alpha-2=-1$. Therefore by a trivial change of variables on the sphere of radius $r$ and using the fact that $n=2$, we get the very same estimate 
$$
\int_{B_{R^2}} |y|^\beta \, |\nabla U|^2 \|A\|^2 \leq \frac{C}{\ln (R)} \xrightarrow{R\to \infty} 0. 
$$
The rest of the proof follows verbatim \cite{deSilvaSavinJEMS}, page $1318$ since this is only based on energy considerations and we refer the reader to it. 
\end{proof}

\section{Uniform bounds around the free boundary}\label{s:UniformBounds}
The optimal regularity bound and the non-degeneracy described in Theorem \ref{theoCRS} were obtained in \cite{CaffarelliRoquejoffreSire} with bounds that depend on the seminorm $\norm{u}_{\dot \HH^\beta(B_1)}$. As a consequence, this dependence propagates to many of our estimates above. 
In this chapter we use the semi-norm dependent estimates (e.g. Lemma \ref{lemFinite}) to prove semi-norm \emph{independent} non-degeneracy estimates. Re-running the arguments above yields the  semi-norm independent results presented in our main Theorem \ref{t:mainregularity}. 

The question of semi-norm independence may seem purely technical; however, independence allows the compactness arguments of the next section to work without additional assumptions on the minimizers involved. 

\subsection{Uniform non-degeneracy}\label{s:UniformNondegeneracy}
We will begin by showing uniform non-degeneracy  from scratch to deduce uniform H\"older character from this fact, reversing the usual arguments in the literature. 

The following lemma was shown in \cite[Corollary 4.2]{Allen} in a more general setting. Here we give a more basic approach based on \cite[Lemma 3.4]{AltCaffarelli}. The main difference is that where Alt and Caffarelli could use the energy to directly control the $H^1$ norm of the minimizer, in our case we need to find an alternative because the measure term of the functional is computed on the thin phase (as opposed to the $H^1$ norm which is computed on the whole space). To bypass this difficulty we will use Allen's monotonicity formula.

The drawback of our approach is that we need the ball to be centered on the free boundary, while in the original lemma, Alt and Caffarelli could center the ball in the zero phase, allowing for a slightly better result.
\begin{lemma}\label{lemNonDegeneracy}
Let $u$ be a minimizer in $B_r$ with $0\in F(u)$. Then $\sup_{\partial B_r} u  \geq Cr^\alpha$ with $C$ depending only on $n$ and $\alpha$.  
\end{lemma}
\begin{proof}
By rescaling we can assume that $r=1$. 

Let $\mathcal{L}u := -\nabla\cdot (|y|^\beta \nabla u)$, consider $\Gamma(x)=\frac{1}{|x|^{n-2\alpha}}$ which is a solution of $\mathcal{L}\Gamma=0$ away from the origin (or $\Gamma(x)=\log |x|$ if $n=1$ and $\alpha=1/2$), and let 
$$\widetilde{v}(x):=\ell  \frac{\max\{1-\Gamma(2x),0\}}{1-\Gamma(2)},$$
where 
$$\ell:=\sup_{\partial B_1} u.$$

It follows that $u\leq \widetilde{v}$ on $\partial B_1$ and thus
$$\mathcal{J}(u, B_1 )\leq \mathcal{J}(\min\{u,\widetilde{v}\},B_1),$$
 and observing that $\tilde{v} = 0$ on $B_{1/2}$ and $\tilde{v} > 0$ on the annulus $A:=B_{1}\setminus B_{1/2}$, we get
\begin{align*}
\int_{B_{\frac12}}|y|^\Beta |\nabla u|^2 + m\left(B_{\frac12,+}(u)\right)
\nonumber	& \leq \int_{A} |y|^\Beta (|\nabla (\min\{u,\widetilde{v}\})|^2-|\nabla u|^2) + m(A_{+}'(\min\{u,\widetilde{v}\})) - m(A_{+}'(u)) \\
\nonumber	& \leq -2\int_{A} |y|^\beta \nabla\max\{u-\widetilde{v},0\}\cdot \nabla \widetilde{v}.
 \end{align*}
By Green's theorem,  writing $d\sigma=|y|^\beta d\mathcal{H}^n$ we get
\begin{align}\label{eqBoundByEll}
\int_{B_{\frac12}}|y|^\Beta |\nabla u|^2 + m\left(B_{\frac12,+}(u)\right)
 \leq -2\int_{\partial B_{\frac12}}  u \partial_\nu \widetilde{v}\, d\sigma 
	 = C_{n, \alpha} \ell \int_{\partial {B_{\frac12}}} u\, d\sigma,
\end{align}
with $C_{n,\alpha} > 0$. 


Using the monotonicity formula and  Proposition \ref{propoAlmostGap}, we get that $\psi^u(r) \geq \psi^u(0)\geq \frac{\omega(B_1)}{2}$ and, therefore
\begin{equation}\label{eqControlJBelow}
\frac{\alpha}{r} \int_{\partial {B_r}} u^2\, d\sigma + \frac{\omega(B_1) r^d}{2}\leq \mathcal{J}_r(u),
\end{equation}
so using H\"older's inequality and the AM-GM inequality we obtain 
\begin{equation}\label{eqControlByMoreThanHolder}
\int_{\partial {B_{\frac12}}} u\, d\sigma\leq \left(\int_{\partial {B_{\frac12}}} u^2\, d\sigma\right)^\frac12 C_{n,\alpha}^\frac12\leq \frac12 \int_{\partial {B_{\frac12}}} u^2\, d\sigma + \frac12 C_{n,\alpha} \leq C_{n,\alpha} \mathcal{J}_{\frac12}(u).
\end{equation}

Combining  \rf{eqBoundByEll}, \rf{eqControlJBelow} and \rf{eqControlByMoreThanHolder} we obtain
$$ 0 < \mathcal{J}_{\frac12}(u)\leq C_{n, \alpha} \ell  \mathcal{J}_{\frac12}(u),$$
and therefore $\ell\geq C_{n,\alpha}^{-1}$.

\end{proof}

To show averaged non-degeneracy we need a mean value principle which is well-known, but we include its proof for the sake of completeness.

\begin{lemma}[Mean value principle]\label{lemMVP}
Let $u\in H^1(\beta,\Omega)$ be a  weak solution to $\mathcal{L}u:=\nabla\cdot (|y|^\beta \nabla u)=0$ in $\Omega$, and let $x_0\in \R^n \times\{0\}$ with $B_r(x_0) \subset \Omega$. Then 
$$u(x_0)=\fint_{B_r}  u \, d\omega$$
where the mean is taken with respect to the measure $d\omega:= |y|^\beta \, dx$.
\end{lemma}

\begin{proof}
Changing variables, we have that
\begin{equation*}
A(\rho):=\frac{1}{\rho^{\beta+n+1}}\int_{B_\rho(x_0)} |y|^\beta u(x) dx=\int_{B_1} |y|^\beta u(\rho x+x_0) dx.
\end{equation*}
On the other hand, set
\begin{align*}
\widetilde{A}(\rho)	
		& := \int_{B_1} |y|^\beta \nabla u(\rho x + x_0)\cdot x \,dx\\
		& = \int_{B_\rho(x_0)} \left(\frac{|y|}{\rho}\right)^\beta  \frac{\nabla u(x)\cdot (x-x_0)}{\rho} \frac{dx}{\rho^{n+1}}
		= \frac{1}{2 \rho^{\beta+n+2}} \int_{B_\rho(x_0)} |y|^\beta  \nabla u(x)\cdot \nabla |x-x_0|^2 \, dx.
\end{align*}
Since $u$ is a weak solution to $\nabla\cdot (|y|^\beta \nabla u)=0$ in $\Omega$, we can apply Green's formula twice to obtain
\begin{align*}
\widetilde{A}(\rho)
	&= \frac{1}{2 \rho^{\beta+n+2}} \int_{\partial B_\rho(x_0)}  |x-x_0|^2 |y|^\beta  \nabla u(x)\cdot \nu \, dx = \frac{1}{2 \rho^{\beta+n}} \int_{\partial B_\rho(x_0)}  |y|^\beta  \nabla u(x) \cdot \nu \, dx=0.
\end{align*}
Since $u$ is absolutely continuous on lines (see \cite[Theorem 4.21]{EvansGariepy}), for almost every $x$ we have $\int_\rho^r \nabla u(t x + x_0)\cdot x \, dt=u(rx+x_0)-u(\rho x+x_0)$. Applying Fubini's Theorem we get
\begin{equation*}
\int_\rho^r \widetilde{A}(t) \, dt = \int_{B_1}|y|^\beta \int_\rho^r \nabla u(t x + x_0)\cdot x \, dt \,dx = \int_{B_1}|y|^\beta  (u(rx+x_0)-u(\rho x+x_0))\,dx = A(r)-A(\rho) .
\end{equation*}
So
$A(r)-A(\rho)= 0$ for all $\rho<r$.

On the other hand, taking the mean with respect to the measure $d\omega:= |y|^\beta \, dx$ and using the continuity of $u$ (see \cite[Theorem 2.3.12]{FabesKenigSerapioni})
we obtain 
$$\left|u(x_0) -\frac{1}{\omega(B_1)}\lim_{\rho \to 0} A(\rho) \right| =  \lim_{\rho \to 0} \frac{1}{\omega(B_\rho(x_0))} \left|\int_{B_\rho(x_0)} (u(x_0) -u(x)) \, d\omega(x)\right| \leq  
\lim_{\rho \to 0} o_{\rho\to 0} (1)=0.$$
\end{proof}

\begin{corollary}\label{coroMeanBoundedBelow}
Let $u$ be a minimizer in $B_r$ with $0\in F(u)$ and let $d\sigma=|y|^\beta d\mathcal{H}^n$. Then $\fint_{\partial B_r} u \, d\sigma \geq Cr^\alpha$ with $C$ depending only on $n$ and $\alpha$. 
\end{corollary}
\begin{proof}
Let $v$ be the $\mathcal{L}-$harmonic replacement of $u$ in $B_r$, that is, the solution to
\begin{equation}\label{eqHarmonicReplacement}
\begin{cases}
\mathcal{L}v=0 & \mbox{ in } B_{r}, \\
v\equiv u &\mbox{ on }\partial B_{r},
\end{cases}
\end{equation}
see \cite[Theorem 3.17]{HeinonenKilpelainenMartio}. After differentiating with respect to the radius, by the mean value principle we get that $v(0)=\fint_{\partial B_r} u \, d\sigma$. By the comparison principle and the Harnack inequality we get that 
\begin{equation}\label{eqComparisonAndHarnack}
C r^\alpha\leq \sup_{B_{r/2}} u \leq \sup_{B_{r/2}} v \leq C \fint_{\partial B_r} u \, d\sigma.
\end{equation}
\end{proof}


\subsection{Behavior of the distributional fractional Laplacian}\label{ss:growth}


Next we use an idea of \cite{AltCaffarelli} and investigate the behavior of the distributional $\alpha$-Laplacian of the minimizer introduced in Section \ref{s:Compactness}. As mentioned in the introduction, in \cite{AltCaffarelli} this investigation immediately yields that the positivity set is a set of locally finite perimeter, and more precisely, that it is Ahlfors regular of the correct dimension. However, the nonlocal nature of this problem indicates that the distributional fractional Laplacian may not be supported on the free boundary and thus we cannot expect to immediately gain such strong geometric information. 


%
%

First we can bound the growth of the fractional Laplacian measure around a free boundary point. Note that this growth is the natural counterpart to the upper Ahlfors regularity in the case of Alt-Caffarelli minimizers.
\begin{theorem}\label{theoLambdaBoundAbove}
Let $u \in \HH^\beta(B_{2r}(x_0))$ be a minimizer of $\mathcal{J}$ in $B_{2r}(x_0)$, and let $x_0\in F(u)$. 
Then, we have
$$\lambda(B_r(x_0)) \leq C r^{n-\alpha}.$$
 In particular $\lambda(F(u))=0$. 
\end{theorem}
A glance at \eqref{eqLaplacianTrivial} will convince the reader that these estimates are sharp, for they cannot be  improved even in the case of the trivial solution. 

\begin{proof}
Without loss of generality we may assume that $x_0=0$.  Let $\mathcal{L}u := -\nabla\cdot (|y|^\beta \nabla u)$ and let $v$ be the $\mathcal{L}$-harmonic replacement of $u$ in $B_{2r}$, see \rf{eqHarmonicReplacement}. Write   $d\sigma=|y|^\beta d\mathcal{H}^n$ and $M:=\fint_{\partial B_{2r}} u \, d\sigma$. By Harnack's inequality (see \cite{CaffarelliRoquejoffreSire}, for instance) and the mean value principle in Lemma \ref{lemMVP},
$$\inf_{B_r} v\geq C v(0)=CM.$$
We have that
$$\lambda(B_r)=\int_{B_r}d\lambda\leq \frac{1}{CM}\int_{B_r}v d\lambda.$$
Since $u\equiv 0$ in the support of  $\lambda$ and $u$ is $\mathcal{L}$-subharmonic  (see \cite[Lemma 2.2]{AltCaffarelli}) we get
$$\int_{B_r}v d\lambda=\int_{B_r}(v-u)d\lambda \leq \int_{B_{2r}}(v-u)d\lambda .$$
By the properties of the measure $\lambda$, we obtain
$$\int_{B_{2r}}(v-u)d\lambda= -\int_{B_{2r}}|y|^\beta \nabla (v-u)\cdot\nabla u =  \int_{B_{2r}}|y|^\beta \left(|\nabla u|^2-|\nabla v|^2\right),$$
and using the definition of the functional and the fact that $u$ is a minimizer, we get
$$\int_{B_{2r}}|y|^\beta \left(|\nabla u|^2-|\nabla v|^2\right) = \mathcal{J}(u,B_{2r})-m(B_{2r}^+(u)) -\mathcal{J}(v,B_{2r})+m(B_{2r}') \leq C r^n.$$
All together, we have that
$$\lambda(B_r)\leq \frac{1}{CM}Cr^n,$$
and, since uniform non-degeneracy (see Lemma \ref{lemNonDegeneracy}) implies that $M\geq Cr^\alpha$ we can conclude the proof of the first statement.

To show the second one, note that since the free boundary has locally finite $(n-1)$-dimensional Hausdorff measure, given a set $E\subset F(u)$ and $k\in\N$ we can find a collection of balls $I_k=\{B^k_i\}_i$ such that 
$$E\subset \bigcup_{B\in I_k} B, \quad\quad \sup_{B\in I_k}r(B)\leq 1/k \quad\quad \mbox{and}\quad\quad \sum_{B\in I_k} r(B)^{n-1}\leq 2\mathcal{H}^{n-1}(E).$$
Thus, 
$$\lambda(E)\leq\sum_{B\in I_k} \lambda(B)\lesssim \sum_{B\in I_k} r(B)^{n-\alpha}\leq \sup_{B\in I_k}r(B)^{1-\alpha} \sum_{B\in I_k} r(B)^{n-1} \xrightarrow{k\to \infty} 0.$$

\end{proof}

 Next we study the measure away from the free boundary. We should emphasize here that even though the estimates in Lemma \ref{lemFracLaplacianIn0Phase} and Theorem \ref{theoMeasureAway} depend on $E_0$, they will be used to remove the dependence of our other estimates on $E_0$. More precisely, Theorem \ref{theoMeasureAway} will play a role in establishing the continuity of the Green function in Lemma \ref{lemContinuousGreen}.  This qualitative fact is used to prove  the quantitative uniform H\"older character in Theorem \ref{theoDimension}. 

After proving Theorem \ref{theoDimension}, we may drop the hypothesis $\norm{u}_{\HH^\beta(B_2)}\leq E_0$  from both Lemma \ref{lemFracLaplacianIn0Phase} and Theorem \ref{theoMeasureAway}.

\begin{lemma}\label{lemFracLaplacianIn0Phase}
If $u \in \HH^\beta_{\loc}(B_{2})$ is a minimizer of $\mathcal{J}$ in the ball $B_{2}$ with $\norm{u}_{\HH^\beta(B_{2})}\leq E_0  $ and $0\in F(u)$, then for every $x_0=(x',0)\in B_{1,0}(u)$ we get
$$\lim_{y\to 0} |y|^\beta |u_y(x',y)| \approx C \dist(x_0,F(u))^{-\alpha}.$$
Moreover, for every ball $B$ centered at $\R^n\times\{0\}$ with $B'\subset\subset B_{1,0}(u)$, we have that
$$|y|^\beta |u_y(x',y)| \leq C \dist(x,F(u))^{-\alpha},$$
for $|y|<C_B\dist(x,F(u))$, where the constant $C_B$ may depend on $B$.
\end{lemma}
\begin{proof}
Let $u$ be a minimizer, and let $B:=B_r(x_0)$ with $B'\subset B_{1,0}(u)$. 
 By \cite[Lemma 2.2]{Silvestre}, we can write $u(x',y)=|y|^{1-\beta}g(x') +\mathcal{O}(y^2)$, where $g$ is a $C^{1+\beta}(\frac12B')$ function, with a uniform control on the error term in terms of  $\norm{u}_{L^2(B,|y|^\beta)}$. In particular, $\lim_{y\to 0} |y|^{\beta-1} u(x',y)=g(x')$.

Let us define 
\begin{equation}
\widetilde{u}(x',y):=\begin{cases}
u(x',y) &\mbox{if }y\geq 0\\
-u(x',-y) &\mbox{if }y< 0.
\end{cases}
\end{equation}
It is clear that $\mathcal{L}\widetilde{u}\equiv 0$ in $B$. According to  \cite[Lemma 3.26, Corollary 3.29]{VitaThesis} $v(x',y)=|y|^\beta y^{-1} \widetilde{u}(x',y)$ is an even $C^\infty(\frac12B)$ function in $\HH^{2-\beta}(B)$ (note that $1<2-\beta<3$ is out of the usual range of $\beta$) and satisfying $\nabla\cdot(|y|^{2-\beta}\nabla v)=0$. The mean value principle (see Lemma \ref{lemMVP}) applies also to this case, so
$$g(x_0')=v(x_0)= \frac{1}{\int_{\frac12B}|y|^{2-\beta}}\int_{\frac12B}|y|^{2-\beta} v(x)  = C \frac{1}{r^{2-\beta+n+1}} \int_{\frac12B}|y| u(x),$$
and using \emph{P1}-\emph{P3} , if $r=\dist(x_0,F(u))$ we get
$$g(x_0')=v(x_0)\approx C r^{\beta-2+1+\alpha}=Cr^{-\alpha}.$$

On the other hand, on the upper half plane we have $u_y=(y^{1-\beta}v)_y=(1-\beta) y^{-\beta} v + y^{1-\beta}v_y$, so
$$y^\beta u_y(x',y)=(1-\beta) v(x',y) + y v_y(x',y),$$
and 
$$\lim_{y\to 0^+} y^\beta u_y(x',y)  = (1-\beta) g(x')\approx r^{-\alpha},$$
the limit being uniform on compact subsets of $B$.
\end{proof}

\begin{theorem}\label{theoMeasureAway}
If $u \in \HH^\beta_{\loc}(B_2)$ is a minimizer of $\mathcal{J}$ in the ball $B_2$ with $\norm{u}_{\HH^\beta(B_2)}\leq E_0$,
then the measure $\lambda$ is absolutely continuous with respect to the Lebesgue measure, and for $m$-almost every $x\in B_{1}'(u)$ we have that
$$\frac{d\lambda}{dm}(x)=2\lim_{y\to0}|y|^\beta u_y(x',y)\approx \chi_{B_{1,0}(u)}(x) \dist(x,F(u))^{-\alpha},$$
with constants depending on $n$, $\alpha$ and $E_0$.
\end{theorem}
\begin{proof}
By Theorem \ref{theoLambdaBoundAbove} we only need to show absolute continuity in  $B_{1,0}(u)\cup B_{1,+}'(u)$.
For $x=(x',0)\in B_{1,+}'(u)$ by \cite[Lemma 4.2]{CaffarelliSilvestre} we have that
$$\lim_{y\to 0}|y|^\beta u_y(x',y)=0,$$
and, for $x\in B_{1,0}(u)$ we have seen in Lemma \ref{lemFracLaplacianIn0Phase} that
$$\lim_{y\to 0}|y|^\beta u_y(x',y)\approx  \dist(x,F(u))^{-\alpha},$$
showing the second part of the statement.



Consider a ball $B_r(x_0)$ with $x_0\in \R^n\times\{0\}$ and a collection of even smooth functions $\chi_{B_r}\leq \psi_k\leq \chi_{B_{r+\frac1k}}$. Then
\begin{equation}\label{eqSandwich}
\lambda(B_r)\leq -\int |y|^\beta \nabla u\cdot \nabla \psi_k\leq \lambda(B_{r+\frac1k}),
\end{equation}
and for every $\varepsilon>0$ we use the Green's theorem to get
$$-\int |y|^\beta \nabla u\cdot \nabla \psi_k= -\int_{|y|\leq\varepsilon} |y|^\beta \nabla u\cdot \nabla \psi_k - \int_{|y|=\varepsilon } |y|^\beta \psi_k \nabla u \cdot \nu\, dm.$$
Using the symmetry properties and taking limits,
\begin{equation}\label{eqControlByUy}
-\int |y|^\beta \nabla u\cdot \nabla \psi_k =2 \lim_{\varepsilon\to 0} \int \varepsilon^\beta \psi_k(x',\varepsilon)  u_y (x',\varepsilon)\, dm(x').
\end{equation}

Next we want to apply the dominated convergence theorem. 
Let us begin by considering a ball $B_r(x_0)\subset B_1$ centered in the zero phase, with $\dist(B'_r(x_0),F(u))\geq 2r$. In this case, by Lemma \ref{lemFracLaplacianIn0Phase} we have
\begin{equation}\label{eqDominateZero}
\varepsilon^\beta   u_y(x',\varepsilon)\lesssim  r^{-\alpha},
\end{equation}
with constants depending perhaps on $u$ and $B_r$ as well.

If instead $B'_r(x_0)\subset\subset B_{1,+}'(u)$, by \cite[Theorem 3.28]{VitaThesis} $u$ is an even $C^\infty$ function on $B'_{r}(x_0)$, so $|y|^\beta u_y=\mathcal{O}(|y|^{1+\beta}).$ Thus
\begin{equation}\label{eqDominatePositive}
\varepsilon^\beta   u_y(x',\varepsilon)\lesssim  r^{2-2\alpha} .
\end{equation}

In both cases, the dominated convergence theorem applies and
$$\lim_{\varepsilon\to 0} \int_{B_{r+\frac1k}\cap \{y=\varepsilon\} } \varepsilon^\beta \psi_k  u_y \, dm= \int_{B_{r+\frac1k}'} \psi_k \lim_{\varepsilon\to 0} (\varepsilon^\beta  u_y)\, dm,$$
and by \rf{eqSandwich} and \rf{eqControlByUy} we obtain
$$\lambda(B_r)\leq 2 \int_{B_{r+\frac1k}'} \psi_k \lim_{\varepsilon\to 0} (\varepsilon^\beta  u_y) \, dm\leq \lambda(B_{r+\frac1k}).$$
In particular $\lim_{\varepsilon\to 0} (\varepsilon^\beta  u_y)\in L^1_{\loc}(B_{1,0}(u)\cup B_{1,+}'(u))$  and taking limits in $k$ we get
$$\lambda(B_r)= 2 \int_{B_{r}'} \lim_{\varepsilon\to 0} (\varepsilon^\beta  u_y) \, dm.$$
\end{proof}

A consequence of our control of the behavior of $\lambda$ is that we can establish the existence of exterior corkscrews. We should note that exterior corkscrews can be also obtained by a purely geometric argument given the non-degeneracy and positive density of Theorem \ref{theoCRS} (see, e.g. the proof of Proposition 10.3 in \cite{davidtoroalmostminimizers}). 
\begin{corollary}\label{coroExteriorCorkscrew}
If $u\in $ is a minimizer in $B_2$ with $\norm{u}_{\HH^\beta(B_2)}\leq E_0$, then $B_{1,+}'(u)$ satisfies the exterior corkscrew condition, i.e. there exists a constant $C_1$ such that for every $x\in F(u)$ and every $0<r<\dist(x,\partial B_1)$ one can find $x_0\in B_r(x)$ so that
$$B(x_0,C_1r)\cap B_{1,+}' (u)= \emptyset.$$
\end{corollary}

\begin{proof}
This is a consequence of Theorems \ref{theoLambdaBoundAbove} and \ref{theoMeasureAway}, and the positive density condition for the zero phase. Indeed, given a ball $B_r\subset \R^{n+1}$, combining both theorems we get
$$\begin{aligned}r^{n-\alpha}
	& \gtrsim \lambda(B_{1,0}(u)\cap B_r) \geq C_{E_0} \int_{B_{1,0}(u)\cap B_r}\dist(x,\partial B_1)^{-\alpha}\\
	&\geq C \left(\sup_{B_{1,0}(u)\cap B_r} \dist(x,\partial B_1)\right)^{-\alpha} |B_{1,0}(u)\cap B_r |,
\end{aligned}$$
and the positive density condition implies that
$$|B_{1,0}(u)\cap B_r | \geq C_{E_0} r^n  .$$
Thus, $$ \sup_{B_{1,0}(u)\cap B_r}  \dist(x,\partial B_1)\geq C_{E_0}r,$$
which is equivalent to the exterior corkscrew condition.
\end{proof}

\subsection{Uniform H\"older character}\label{s:UniformFinal}

The uniform non-degeneracy of Section \ref{s:UniformNondegeneracy} lets us conclude uniform control on the H\"older norm of $u$. 

\begin{theorem}\label{theoDimension}
Let $u$ be a minimizer of $\mathcal{J}$ in $B_r$ with $0\in F(u)$. Then $|u(x)|\leq C|x|^\alpha$ for every $x\in \partial B_{r/2}$ with $C$ depending only on $n$ and $\alpha$.
\end{theorem}
\begin{proof}
Again we set $v$ to be the $\mathcal{L}$-harmonic replacement of $u$ inside of $B_r$ as in \rf{eqHarmonicReplacement}. Let $\widetilde{u}:=v- u$, so that
$$\mathcal{L}\widetilde{u} = \mathcal{L}v - \mathcal{L}u = -\lambda = - \nabla\cdot (|y|^\beta \nabla u),$$
and $\widetilde{u}\in H^{1,2}_0(B_r; |y|^\beta)$. 

Consider the Green function $G: B_r\times B_r\to \R$ such that $\mathcal{L} G(\cdot , z) = \delta_z$, and $G(\cdot , z)\in H^{1,2}_{\loc}(\overline{B_r}\setminus\{z\})$ with null trace on $\partial B_r$
(see \cite[Proposition 2.4]{FabesJerisonKenig}).
By \cite[Proposition 2.1, Lemma 2.7]{FabesJerisonKenig} there exists $p_0>1$ so that $\widetilde{u}$ is the unique function in $ H^{1,p_0}_0(B_r;|y|^\beta)$ such that $\mathcal{L}\widetilde{u}=\lambda$, and moreover
\begin{equation}\label{eqGrennRules}
\widetilde{u}(z)=\int_{B_r} G(z,x)\, d\lambda(x),
\end{equation}
for almost every $z\in B_r$.

Below, in Lemma \ref{lemContinuousGreen}, we will see that the equality \rf{eqGrennRules} is in fact valid for every $z\in B_{r/4}$, that is, 
$\widetilde{u}=\int_{B_r} G(\cdot,x)\, d\lambda(x)$. In particular
$$v(0)=\widetilde{u}(0)=\int_{B_r} G(0,x)\, d\lambda(x).$$

Next we use the following estimate (see \cite[Theorem 3.3]{FabesJerisonKenig}): let $z,x\in B_{r/4}$. Then
$$G(z,x)\approx \int_{|x-z|}^r \frac{s \, ds}{w(B(x,s))},$$
where $w$ is the $A_2$ weight $w(x)=|y|^\beta$. Computing, for $x=(x',y)$ we obtain
$$w(B(x,s))\approx s^n \int_{y-s}^{y+s} |t|^\beta\, dt \approx s^n \max\{|y|,s\}^{\beta+1}.$$
First we assume that $n-2\alpha>0$.
Thus, if $x\in B_{r/4}'$ then
\begin{equation}\label{eqGreenEstimate}
G(z,x) \approx \int_{|x-z|}^r s^{-n-\beta} \, ds \approx |x-z|^{-n-\beta+1} = |x-z|^{2\alpha-n}.
\end{equation}

 Note that $\lambda(B_r) \leq Cr^{n-\alpha}$ by Theorem \ref{theoLambdaBoundAbove}. Thus, writing $A_{t,s}:=B_s\setminus B_t$, we have that
 $$ v(0)= \int_{B_r} G(0,x)\, d\lambda(x) \leq \int_{cr^{2\alpha-n}}^\infty \lambda\left(\{x\in B_{r/4}: G(0,x) > t\}\right) dt + \int_{A_{r/4,r}} G(0,x)\, d\lambda(x).$$
By the strong maximum principle, the Green function in the annulus is bounded by $C r^{n-2\alpha}$. This fact, together with Theorem \ref{theoLambdaBoundAbove}, implies that
 $$ v(0) \leq \int_{cr^{2\alpha-n}}^\infty \lambda\left(B_{Ct^{\frac{-1}{(n-2\alpha)}}}\right) dt +Cr^\alpha \leq C \int_{cr^{2\alpha-n}}^\infty t^{-\frac{n-\alpha}{n-2\alpha}} dt +Cr^\alpha = Cr^\alpha.$$
 By the mean value theorem we conclude that
$$\fint_{\partial B_r} v \, d\sigma \leq C r^{\alpha},$$ where $d\sigma=|y|^\beta d\mathcal{H}^n$. 
The theorem follows by observing that, as in \rf{eqComparisonAndHarnack}, the mean of $v$ dominates $u$ by $\sup_{\partial B_{r/2}} u \leq \sup_{\partial B_{r/2}} v \leq C \fint_{\partial B_r} v\, d\sigma$.

In case $n-2\alpha=0$, which could only happen for $n=1$ and $\alpha=1/2$, estimate \rf{eqGreenEstimate} reads as
\begin{equation*}
G(z,x) \approx  \log\left(\frac{r}{|x-z|}\right),
\end{equation*}
and the proof follows the same steps.

In case $n-2\alpha<0$, then  estimate \rf{eqGreenEstimate} reads as
\begin{equation*}
G(z,x) \approx  r^{n-2\alpha},
\end{equation*}
and the estimate is even better compared to the above.

\end{proof}
%
%

\begin{lemma}\label{lemContinuousGreen}
$\int_{B_r} G(z,x)\, d\lambda(x)$ is continuous in $z\in B_{r/4}$.
\end{lemma}
\begin{proof}
Let $\varepsilon<r/2$ and let $z_1, z_2\in B_{r/4}$,  with $|z_1-z_2|\leq \varepsilon/2$. Then 
\begin{align}\label{eqBreakGMinusG}
\int_{B_r}|G(z_1,x)-G(z_2,x)|\, d\lambda(x)
\nonumber	& \leq \int_{B_r\setminus B_\varepsilon(z_1)}|G(z_1,x)-G(z_2,x)|\, d\lambda(x) \\
	& \quad + \int_{B_\varepsilon(z_1)}G(z_1,x)d\lambda(x) + \int_{B_\varepsilon(z_1)}G(z_2,x)d\lambda(x).
\end{align}

Next we use \rf{eqGreenEstimate} and Theorems \ref{theoLambdaBoundAbove} and \ref{theoMeasureAway}. By decomposing the domain on dyadic annuli, in case $n-2\alpha>0$ we get
\begin{align}\label{eqBoundGreenLambda}
\int_{B_\varepsilon(z_1)}G(z_1,x)d\lambda(x)
\nonumber	& \leq \sum_{j\leq0}\int_{A_{2^{j-1}\varepsilon,2^{j}\varepsilon}(z_1)}G(z_1,x)d\lambda(x)
		\lesssim \sum_{j\leq0} \lambda(B_{2^j\varepsilon}(z_1))(2^{j-1}\varepsilon)^{2\alpha-n}  \lesssim\varepsilon^\alpha \sum_{j\leq0} 2^{j\alpha}.
\end{align}
In case $n-2\alpha=0$ we obtain $\varepsilon^\alpha\sum_{j\leq0} 2^{j\alpha}\log\left(\frac{r}{2^j\varepsilon}\right)$ on the right-hand side instead, and in case $n-2\alpha<0$ we obtain $\varepsilon^{n-\alpha}r^{2\alpha-n}\sum_{j\leq0} 2^{j(n-\alpha)}$. In every case, fixing $\varepsilon$ small enough this term can be as small as wanted. The same will happen with the last term on the right-hand side of \rf{eqBreakGMinusG}.

On the other hand, by \cite[Theorem 2.3.12]{FabesKenigSerapioni} Green's function is uniformly continuous on the set $\{(z,x)\in B_r\times B_r: |z-x|>\varepsilon \}$ so $|G(z_1,x_1)-G(z_2,x_2)|\leq \delta_\varepsilon(|z_1-z_2|+|x_1-x_2|)$ with $\delta_\varepsilon(t)\xrightarrow{t\to 0} 0$. Thus, 
$$\int_{B_r\setminus B_\varepsilon(z_1)}|G(z_1,x)-G(z_2,x)|\, d\lambda(x) \leq \delta_\varepsilon(|z_1-z_2|)\lambda(B_r)\to 0 .$$
Assuming that $|z_1-z_2|$ is small enough, we obtain that $\int_{B_r}|G(z_1,x)-G(z_2,x)|\, d\lambda(x)$ is as small as wanted and the claim follows.
\end{proof}

\begin{remark}\label{remUniformRules}
 In light of Theorem \ref{theoDimension}, arguing as in \cite[Theorem 1.1]{CaffarelliRoquejoffreSire} we obtain that every minimizer $u$ in a ball $B_r$ with $0\in F(u)$ has uniform $C^\alpha$ character in $B_{r/2}$. By the Caccioppoli inequality (see Section \ref{secCaccioppoli}) we also obtain the same for the $\HH^\beta$ norm. Moreover, using \cite[Theorem 1.2]{CaffarelliRoquejoffreSire} we can find interior corkscrew points with constants not depending on these norms. This allows us to remove the {\it a priori} dependence on $\|u\|_{\HH^\beta}$ from all of our results above. 
\end{remark}

\subsection{Lower estimates for the distributional fractional Laplacian}
Next we bound the growth of the measure around a free boundary point from below. None of these results will be used in the present paper, but we include them to give a complete picture of the tools under consideration.
\begin{theorem}\label{theoMeasureBall}
Let $u \in \HH^\beta(B_{2r})$ be a minimizer of $\mathcal{J}$ in $B_{2r}$ such that $0\in F(u)$. Then we have
$$\lambda(B_r) \geq C r^{n-\alpha}.$$
\end{theorem}

\begin{proof}
Let $\mathcal{L}u := -\nabla\cdot (|y|^\beta \nabla u)$ and let $v$ be the $\mathcal{L}$-harmonic replacement of $u$ in $B_r$ (see \rf{eqHarmonicReplacement}). Let $\widetilde{u}:=v- u$ and
consider the Green function $G: B_r\times B_r\to \R$  as in the proof of Theorem \ref{theoDimension}.

Let $0<\kappa<1$ to be fixed later. By \emph{P1}-\emph{P3} in Theorem \ref{theoCRS} there exists a point $z_0\in B_{\kappa r}$ with 
\begin{equation}\label{eqU}
u(z_0)\approx (\kappa r)^\alpha,
\end{equation}
with constants depending only on $n$ and $\alpha$ by Remark \ref{remUniformRules}.
By \emph{P1} there is a constant $c$ such that for every $z\in B(z_0,c\kappa r)$  we have that $u(z)\approx (\kappa r)^\alpha$. Since $\lambda$ is supported on the zero phase of $u$, the ball $B(z_0,c\kappa r)$ is away from its support, and
$$\widetilde{u}(z)=\int_{B_r\setminus B(z_0,c\kappa r) } G(z,x)\, d\lambda(x).$$

Using  the strong maximum principle (see \cite[Theorem 6.5]{HeinonenKilpelainenMartio}) and \rf{eqGreenEstimate}, for almost every $z\in B(z_0,c\kappa r/2)$ we get
\begin{align*}
\widetilde{u}(z)
	& \leq \lambda(B_r) \sup_{x\in B_r\setminus B(z_0,c\kappa r) } G(z,x) = \lambda(B_r) \sup_{x\in B_{r/4} \setminus B(z_0,c\kappa r) } G(z,x) \\
	& \approx   \lambda(B_r)\sup_{x\notin B(z_0,c\kappa r) }  |x-z|^{2\alpha-n} =\lambda(B_r) (c \kappa r)^{2\alpha-n}.
\end{align*}
That is,
\begin{equation}\label{eqTildeU}
\widetilde{u}(z)
	 \lesssim \lambda(B_r)  (c\kappa r)^{2\alpha-n}.
\end{equation}

On the other hand, note that $u$ is continuous. By the Riesz representation theorem, there exists a probability measure $\omega^z_\mathcal{L}$ such that
$$v(z)= \int_{\partial B_r} u(x) d\omega^z_\mathcal{L}(x).$$

We can choose $r$ so that $\partial B_r$ intersects a big part of a corkscrew ball, i.e., assume that there exists a point $\xi_0\in \partial B_r'$ which is the center of a ball $B'(\xi_0, cr)$ where $u$ has positive values. This can be done by the interior corkscrew condition, with all the constants involved depending only  on $n$ and $\alpha$. Then, changing the constant if necessary, all points  $\xi\in B(\xi_0, cr)$ satisfy that $u(\xi)\geq C r^\alpha$ by the non-degeneracy condition and the optimal regularity. Call $U:=\partial B_r\cap B(\xi_0, cr)$. Then
$$v(z)\gtrsim r^\alpha \omega^z_\mathcal{L}(U).$$
But $\omega^z_\mathcal{L}(U)$ is bounded below by a constant by \cite[Lemma 11.21]{HeinonenKilpelainenMartio} and the Harnack inequality (use a convenient Harnack chain). All in all, we have that
\begin{equation}\label{eqV}
 v(z)\gtrsim r^\alpha.
\end{equation}

%

 Combining \rf{eqTildeU}, \rf{eqU} and \rf{eqV} and choosing $\kappa$ small enough, depending in $n$ and $\alpha$, we get
$$ \lambda(B_r) \gtrsim \frac{\widetilde{u}(z_0)}{(c \kappa r)^{2\alpha-n}}\geq  \frac{ Cr^\alpha - C' (\kappa r)^\alpha }{(c \kappa r)^{2\alpha-n}} \geq C_{n,\alpha}  r^{n-\alpha} ,$$
for $\kappa$ small enough.

In case $n-2\alpha=0$, that is for $n=1$ and $\alpha=1/2$, using similar changes as in the proof of Theorem \ref{theoDimension}  we get  $\widetilde{u}(z)\lesssim \lambda(B_r)\sup_{x\notin B(z_0,c\kappa r)}  \log\left(\frac{r}{|x-z|}\right) \approx \lambda(B_r)|\log \kappa|$ instead of \rf{eqTildeU}. In case $n-2\alpha<0$, the proof is even easier than before.
\end{proof}

\begin{remark}\label{remDimension}
Theorem \ref{theoMeasureBall} implies that the $(n-\alpha)$-Hausdorff measure of the free boundary is locally finite. This does not suffice to show finite perimeter of the positive phase and, therefore, we had to use the approach in Section \ref{s:finperimeter}.
\end{remark}

The following theorem summarizes the information that we have gathered so far about the measure $\lambda$.

\begin{theorem}\label{theoMeasureComplete}
If $u \in \HH^\beta_{\loc}(\Omega)$ is a  minimizer of $\mathcal{J}$ in $\Omega$, then the measure $\lambda$ is absolutely continuous with respect to the Lebesgue measure in $\Omega'(u)$. Moreover, given $x_0\in F(u)$ and $r>0$ such that $B_{2r}(x_0)\subset \Omega$, then
\begin{equation}\label{eqMeasureBallNice}
\lambda(B_r(x_0)) \approx r^{n-\alpha},
\end{equation}
%
and for almost every $x\in B_{r}'(x_0)$ we have that
$$\frac{d\lambda}{dm}(x)=2\lim_{y\to0}|y|^\beta u_y(x',y)\approx \chi_{\Omega_0(u)}(x) \dist(x,F(u))^{-\alpha},$$
with constants depending only on $n$ and $\alpha$.
\end{theorem}

\section{Rectifiability of the singular set}\label{s:Rect}

In this section we use the Rectifiable-Reifenberg and quantitative stratification framework of Naber-Valtorta \cite{NaberValtortaRectifiable} to prove Hausdorff measure and structure results for the singular set. Recall that $k^*_\alpha$ is the first dimension in which there exists non-trivial $\alpha$-homogeneous global minimizers to \eqref{eqLocalizedFunctional} defined in Section \ref{s:finperimeter}. 

\begin{theorem}\label{theoGlobalRectifiability}
Let $u\in \HH^\beta_{\loc}(\Omega)$ be a minimizer of \rf{eqLocalizedFunctional} in a domain $\Omega$. Then $\Sigma(u)$ is $(n-k^*_\alpha)$-rectifiable and for every $D\subset\subset \Omega$, we have
$$\mathcal{H}^{n-k^*_\alpha}(\Sigma(u)\cap D)\leq C_{n,\alpha, \dist(D,\partial\Omega)}.$$
\end{theorem}

Part of the power of this framework is that it is very general. One needs certain compactness properties on the minimizers and a connection between the drop in the monotonicity formula and the local flatness of the singular set (see Theorem \ref{theoL2Estimate} below). To avoid redundancy and highlight the original contributions of this article, we omit many details here and try to focus on the estimates needed to apply this framework to minimizers of \eqref{eqLocalizedFunctional}. Whenever we omit details we will refer the interested reader to the relevant parts of \cite{EdelenEngelstein}. 

The key first step is to introduce the appropriate formulation of quantitative stratification. First introduced by Cheeger and Naber \cite{CheegerNaber} in the context of manifolds with Ricci curvature bounded from below, this is a way to quantify the intuitive fact that $F(u)$ should ``look" $(n-k^*_{\alpha})$-dimensional near a point $x_0 \in F(u)$ at which the blow-ups have $(n-k^*_\alpha)$-linearly independent translational symmetries. 

\subsection{Quantitative stratification for minimizers to $\mathcal J$}

We have seen in Section \ref{secReduction} that homogeneous functions have linear spaces of translational symmetry. Here we want to quantify (both in terms of size and stability) how far a function is from having no more than $k$ directions of translational symmetry. 

\begin{definition}
We write $V^k$ for the collection of linear $k$-dimensional subspaces of $\R^n$. A function $u$ is said to be $k$-symmetric if it is $\alpha$-homogeneous with respect to some point, and there exists a  $L\in V^k$ so that 
$$u(x+v)=u(x), \quad\quad \mbox{for every } v\in L.$$
A function $u$ is said to be $(k,\epsilon)$-symmetric in a ball $B$ if for some $k$-symmetric $\widetilde{u}$ we have
$$r(B)^{-2-n} \int_{B} |y|^\Beta |u-\widetilde{u}|^2 dy <\epsilon.$$
\end{definition}

Next we define the $k$-stratum $S^k(u)$, the $(k,\epsilon)$-stratum $S^k_\epsilon(u)$ and the $(k,\epsilon,r)$-stratum $S^k_{\epsilon,r}(u)$. A key insight here is to define these strata by the blow-ups having $k$ or fewer symmetries as opposed to exactly $k$ symmetries. 
\begin{definition}
Let $0\leq k\leq n$,  $0<\varepsilon<\infty$ and $0<r<\dist(x,\partial\Omega)$, let $u$ be a continuous function in $\Omega$ and let $x\in F(u)$. We say that:
\begin{itemize}
\item $x\in S^k(u)$ if $u$ has no  $(k+1)$-symmetric blow-ups at $x$.
\item $x\in S^k_\epsilon(u)$ if $u$ is not  $(k+1,\epsilon)$-symmetric in $B_s(x)$ for every $0<s\leq \min\{1,\dist(x,\partial\Omega)\}$.
\item $x\in S^k_{\epsilon,r}(u)$ if $u$ is not  $(k+1,\epsilon)$-symmetric in $B_s(x)$ for every $r\leq s\leq \min\{1,\dist(x,\partial\Omega)\}$.
\end{itemize}
If it is clear from the context we will omit $u$ from the notation.
\end{definition}

We now detail some standard properties of the strata defined above and how they interact with the free boundary $F(u)$. While the proofs are mostly standard, we give the details as the scaling associated to the problem \eqref{eqLocalizedFunctional} adds some technical difficulties. This proof also provides a blueprint for fleshing out the details in Sections \ref{secToolsBad} and \ref{secToolsGood}.

\begin{lemma}\label{lemProperties}
Let $0\leq j \leq  k\leq n$,  $0< \varepsilon \leq\tau<\infty$, $0<r\leq s<\dist(x,\partial\Omega)$, and let $u\in \HH^\beta_{\loc}(\Omega)$ be a minimizer in $\Omega$. Then:
\begin{enumerate}
\item $S^0\subset S^1\subset \cdots \subset S^{n-1} = S^n= F(u)$. Moreover, for the reduced boundary, we have that
$F_{red}(u) \subset S^{n-1}\setminus S^{n-2}$ and $\Sigma (u) \subset S^{n-k^*_\alpha}$.
\item We have $S^j_{\tau}\subset S^k_\epsilon \subset S^k$  and, moreover,  $\displaystyle S^k=\bigcup_{\epsilon>0} S^k_\epsilon$. 
\item Also $S^j_{\tau}\subset S^j_{\tau, r}\subset S^k_{\epsilon, s}$ and, moreover,  $\displaystyle S^k_\epsilon= \bigcap_{r>0} S^k_{\epsilon, r}$.
\item The sets $S^k_\epsilon$ are closed, in both $x$ and $u$: if $u_i\xrightarrow{L^2_{\rm loc}(\Omega;|y|^\Beta)} u$ and $x_i\to x$ with $x_i\in  S^k_\epsilon(u_i)$, then $x\in S^k_\epsilon(u)$.
\item If $u_i \xrightarrow{L^2_{\rm loc}(\Omega; |y|^\Beta)} u$,  $\epsilon_i\to 0$, and $u_i$ are $(k,\epsilon_i)$-symmetric in $B_1$, then $u$ is $k$-symmetric in $B_1$.
\end{enumerate}
\end{lemma}
\begin{proof}
\emph{1}. The inclusions $S^k\subset S^{k+1}$ of the first property are trivial. The last equalities are consequences of the non-degeneracy. The fact that $F_{red}(u) \cap S^{n-2} =\emptyset$ can be deduced from the Hausdorff convergence of the free boundaries described in Lemma \ref{lemCompactness} and Theorem \ref{theoImprovement}. Finally, $\Sigma (u) \subset S^{n-k^*_\alpha}$ is a consequence of Lemmas \ref{lemDimensionReduction} and \ref{lemTrivialIsTheChosen}.

\emph{2}. The inclusions $S^j_{\tau}\subset S^k_\epsilon$ of the second property come from the definitions: if $x\notin S^k_\epsilon$ then there exist a ball $B\subset \Omega$ centered at $x$ and a $(k+1)$-symmetric $\widetilde{u}$ so that $r(B)^{-2-n} \int_{B} |y|^\Beta |u-\widetilde{u}|^2 dy <\epsilon\leq \tau.$ But  $\widetilde{u}$ is also $(j+1)$-symmetric. Thus, $x\notin S^j_\tau$. 

The fact that $S^k_\epsilon \subset S^k$ is a consequence of the uniform convergence on Lemma \ref{lemCompactness}: if $x\notin S^k$, then $u$ has a  $(k+1)$-symmetric blow-up sequence $u_i\to u_0$ at $x$ converging uniformly. Thus,
\begin{align*}
 \int_{B_{\rho_i}(x)} |y|^\Beta \left|u(x)- \rho_i^\alpha u_0\left(\frac{x-x_0}{ \rho_i}\right)\right|^2 dx 
 	& =  \rho_i^{\Beta+2\alpha+n+1} \int_{B_1(x)} |y|^\Beta \left|\frac{u\left(x_0+ \rho_i x\right)}{\rho_i^\alpha}-u_0(x)\right|^2 dx\\
 	& \leq \rho_i^{n+2} \omega(B_1) \norm{u_i-u_0}_{L^\infty}.
\end{align*}
That is, 
$$ \rho_i^{-n-2} \int_{B_{\rho_i}(x)} |y|^\Beta \left|u(x)- r_i^\alpha u_0\left(\frac{x-x_0}{ \rho_i}\right)\right|^2 dx \xrightarrow{i\to\infty} 0,$$
and therefore, for every $\varepsilon$ there exists a ball small enough so that $u$ is $(k+1,\varepsilon)$-symmetric in it. In particular $S^k \supset \bigcup_{\epsilon>0} S^k_\epsilon$. 

To see the converse, assume that $x\notin \bigcup_\epsilon S^k_\epsilon$. Then for every $i\in\N$ there exist a $(k+1)$-symmetric function $\widetilde u_i$, invariant with respect to  $L_i\in V^{k+1}$  and $r_i<\min\{1,\dist(x,\partial\Omega)\}$ such that 
$$\frac{1}{r_i^{n+2}} \int_{B_{r_i}} |y|^\Beta \left| u(x)-\widetilde{u}_i(x)\right|^2 \, dx < \frac1i.$$
In the case when  $r_i$ stays away from zero, since $r_i<1$, we can take a subsequence converging to $r_0\in(0,1)$, and one can see that $u$ is $(k+1)$-symmetric in the ball $B_{r_0}(x_0)$. Otherwise, 
consider $u_i:=\frac{u\left(x_0+ r_i x\right)}{r_i^\alpha}$ and $\widetilde{u}_{i,i}=\frac{\widetilde u_i\left(x_0+ r_i x\right)}{r_i^\alpha}$. Taking subsequences, we can assume that $L_i\to L_0$ locally in the Hausdorff distance, and that $u_i \to u_0$ locally uniformly. One can check also using the H\"older character of $u$ that $\{\widetilde{u}_{i,i}\}$ is uniformly bounded in $L^2( B; |y|^\Beta)$, so taking subsequences again, we can assume the existence of $\widetilde{u}_{0}$ so that
 $\widetilde{u}_{i,i}\to \widetilde{u}_{0}$ in $L^2( B; |y|^\Beta)$. This function will be $(k+1)$-symmetric, being invariant in the directions of $L_0$. By the triangle inequality we get
$$\int_{B_1} |y|^\Beta |u_0- \widetilde{u}_0|^2 \, dx \lesssim \int_{B_1} |y|^\Beta |u_0- u_i|^2 \, dx +\int_{B_1} |y|^\Beta |u_i- \widetilde{u}_{i,i}|^2 \, dx +\int_{B_1} |y|^\Beta |\widetilde{u}_{i,i}- \widetilde{u}_0|^2 \, dx .$$
The first and the last integrals converge to zero by our choice of the subsequence. For the middle term just change variables as before:
$$\int_{B_1} |y|^\Beta |u_i- \widetilde{u}_{i,i}|^2 \, dx = \frac{1}{r_i^{n+2}} \int_{B_{r_i}} |y|^\Beta \left| u(x)-\widetilde{u}_i(x)\right|^2 \, dx\to 0.$$
Thus we have that $u_0=\widetilde{u_0}$ and, therefore, $x\notin S_k$.

\emph{3}. The inclusions $S^j_{\tau}\subset S^j_{\tau, r}\subset S^k_{\epsilon, s}$ of the third property come from the definitions and thus, $S^k_\epsilon \subset \bigcap_{r>0} S^k_{\epsilon, r}$. The converse implication is also trivial.

\emph{4}. The closedness is obtained by a contradiction argument again. It is straightforward but we write it here for the sake of completeness.
 
Assume by contradiction that $x\notin S^k_\epsilon(u)$. Then there exist a $(k+1)$-symmetric function $\widetilde{u}$ and a radius $r$ such that
$$\epsilon_0:=\frac{1}{r^{n+2}} \int_{B_{r}(x)} |y|^\Beta \left| u(x)-\widetilde{u}(x)\right|^2 \, dx <\epsilon.$$
Let $\tau<1$ to be fixed and consider  $i_0\in \N$ so that $B_{\tau r}(x_i)\subset B_r(x)$ for every $i\geq i_0$. By the triangle inequality
$$\frac{1}{(\tau r)^{n+2}} \int_{B_{\tau r}(x_i)} |y|^\Beta \left| u_i(x)-\widetilde{u}(x)\right|^2 \, dx \leq \frac{1}{(\tau r)^{n+2}} \norm{u_i-u} _{L^2(B_{\tau r}(x_i);|y|^\Beta)}^2 + \frac{\epsilon_0}{\tau^{n+2}}.$$
We define $\tau$ so that $\frac{\epsilon_0}{\tau^{n+2}}=\frac{\epsilon+\epsilon_0}{2}$. Choose  $i_0$ big enough so that every $i\geq i_0$ satisfies that $\norm{u_i-u} _{L^2(B_{\tau r}(x_i);|y|^\Beta)}^2 < (\tau r)^{n+2} \frac{\epsilon-\epsilon_0}{2}$. Then $x_i\notin S^k_\epsilon(u_i)$, contradicting the hypothesis. 

\emph{5}. Assume that  $\widetilde u_i$ is invariant with respect to  $L_i\in V^{k+1}$ and
$$\int|y|^\Beta |u_i-\widetilde u_i|^2\leq \epsilon_i.$$ Consider a subsequence $\{u_i\}$ so that the varieties $L_i\to L$ locally in the Hausdorff distance. Using the triangle inequality as in \emph{4} it follows that $u$ is $(k,\delta_i)$-symmetric with $\delta_i\to 0$. 

\end{proof}

\begin{proposition}\label{propoSigma}
There exists $\epsilon(n,\alpha) > 0$ such that if $u\in \HH^\beta_{\loc}(\Omega)$ is a minimizer of $\mathcal{J}$ in a domain $\Omega\subset \R^{n+1}$, then $\Sigma(u)\subset S^{n-k^*_\alpha}_\epsilon(u)$. 
\end{proposition}
\begin{proof}
It is enough to show that if $u$ is a minimizer of $\mathcal{J}$ in $B_2(0)$, then $\Sigma(u)\cap B_1(0)\subset S^{n-k^*_\alpha}_\epsilon(u)$. 

By contradiction, let us assume that there is a sequence of positive numbers $\epsilon_i\xrightarrow{i\to\infty} 0$, functions $u_i$ minimizing $\mathcal{J}$ in $B_2(0)$ and $x_i\in  \Sigma (u_i)\cap B_1(0)$, $r_i \in (0,1]$, with $u_i$ being $(n-k^*_\alpha+1, \epsilon_i)$-symmetric  in $B_{r_i}(x_i)$, and let $L_i$ be an $(n-k^*_\alpha+1)$-dimensional subspace that leaves invariant one of the admissible $(n-k^*_\alpha+1)$-symmetric approximants. By rescaling we can assume that $r_i=1$. 

Passing to a subsequence we can assume that $L_i \to L_0\in V^{n-k^*_\alpha+1}$ locally in the Hausdorff distance and $x_i\to x_0$. By the compactness results in Lemma \ref{lemCompactness} we have a uniform limit $u_0$ which is a minimizer as well, and it is $(n-k^*_\alpha+1)$-symmetric with invariant manifold $L_0$. By Lemma \ref{lemConeBlowup} any blow-up $u_{0,0}$ at $x_0$ will be $(n-k^*_\alpha+1)$-symmetric as well. Applying Lemma \ref{lemDimensionReduction} $(n-k^*_\alpha+1)$ times we find that the restriction of $u_{0,0}$ to the orthogonal manifold $L_0^{\perp}$ is a $(k^*_\alpha-1)$-dimensional minimal cone which, by Lemma \ref{lemTrivialIsTheChosen} is the trivial solution, and so is $u_{0,0}$. Thus, $x_0$ is a regular point for $u_0$.

On the other hand,  the Hausdorff convergence of Lemma \ref{lemCompactness} together with the improvement of flatness of Theorem \ref{theoImprovement} imply that for $i$ big enough $x_i\in F_{\rm red}(u_i)$, reaching a contradiction.
\end{proof}

\subsection{The Refined Covering Theorem}

Our estimates on the size and structure of the singular set $\Sigma(u)$  come from similar results concerning the $S^k_{\epsilon}(u)$. In particular, we prove the following covering result:

\begin{theorem}\label{theoMain}
Let $u\in \HH^\beta(B_5)$ be a  minimizer to \rf{eqLocalizedFunctional} in $B_5$  with   $0\in F(u)$. For given real numbers $\epsilon>0$, $0<r\leq 1$ and every natural number $1\leq k\leq n-1$, we can find a collection of balls $\{B_r(x_i)\}_{i=1}^N$ with $N\leq C_{n,\alpha,\epsilon} r^{-k}$ such that 
$$S^k_{\epsilon,r}(u)\cap B_1\subset \bigcup_i B_r(x_i).$$
In particular, $|B'_r(S^k_{\epsilon,r}\cap B_1)|\leq C_{n,\alpha,\epsilon} r^{n-k}$ for every $0<r\leq 1$ and 
$$\mathcal{H}^k(S^k_{\epsilon}(u) \cap B_1)\leq C_{n,\alpha,\epsilon}.$$
\end{theorem}

From Proposition \ref{propoSigma} and Theorem \ref{theoMain} we can conclude the following corollary which comprises the second part of Theorem \ref{theoGlobalRectifiability} above. 

\begin{corollary}\label{coroSingular}
If $u\in \HH^\beta(B_5)$ is a minimizer to \rf{eqLocalizedFunctional}  in $B_5$ with  $0\in F(u)$, then $\Sigma(u)$ is $(n-k^*_\alpha)$-rectifiable and for every $0<r\leq 1$ we have 
$$|B_r (\Sigma(u)\cap B_1)|\leq C_{n,\alpha} r^{k^*_\alpha}.$$
In particular, 
$$\mathcal{H}^{n-k^*_\alpha}(\Sigma(u)\cap B_1)\leq C_{n,\alpha}.$$
\end{corollary}

Rectifiability is encoded in the following result. We omit the details of proof here but it is a consequence of the packing result above, the Rectifiable-Reifenberg theorem of \cite{NaberValtortaRectifiable} and Theorem \ref{theoL2Estimate} below. For more details see Sections 2 and 8 of \cite{EdelenEngelstein} (particularly Theorem 2.2 in the former and the proof of Theorem 1.12 in the latter). 

\begin{theorem}\label{theoMain2}
Let $u$ be a non-negative, even minimizer to \rf{eqLocalizedFunctional} in a domain $\Omega$. Then $S^k_\epsilon(u)$ is $k$-rectifiable for every $\epsilon$ and, hence, each stratum $S^k(u)$ is $k$-rectifiable as well.
\end{theorem}

The proof of Theorem \ref{theoMain} follows from inductively applying the following, slightly more technical, packing result (for details see Section 4 of \cite{EdelenEngelstein}).

\begin{theorem}\label{theoTreeConstruction}
Let $\epsilon>0$. There exists $\eta (n,\alpha, \epsilon)$ such that, for every minimizer $u\in \HH^\beta(B_5)$ of $\mathcal{J}$ in $B_5$ with $0\in F(u)$ and $0<R<1/10$, there is a finite collection $\mathcal{U}$ of balls $B$ with center $x_B \in S^k_{\epsilon, \eta R}$ and  radius $R \leq r_B \leq 1/10$ which satisfy the following properties:
\begin{enumerate}[A)]
\item Covering control:
$$  S^k_{\epsilon, \eta R} \cap B_1  \subset \bigcup_{B\in \mathcal{U}} B.$$
\item Energy drop:
For every $B\in \mathcal{U}$, 
$$\mbox{either}\quad r_B=R, \quad\quad \mbox{or}\quad \sup_{2B} \Psi^u_{2r_B}\leq \sup_{B_2}\Psi^u_2 -\eta.$$
\item Packing:
$$\sum_{B\in \mathcal{U} } r_B^k \leq c(n, \alpha, \epsilon).$$
\end{enumerate}
\end{theorem}

We construct the balls of Theorem \ref{theoTreeConstruction} using a ``stopping time" or ``good ball/bad ball" argument. Much of this argument uses harmonic analysis and geometric measure theory and is completely independent of the original problem \eqref{eqLocalizedFunctional}. However, there are a few places in which we need to connect the behavior of minimizers to the geometric structure of the singular set. Here we will sketch the ``good ball/bad ball" argument, taking for granted the estimates needed to apply this argument to our functional. In the next few subsection we will provide these estimates. For more details on the construction itself we refer the reader to Section 7 in \cite{EdelenEngelstein}.

\smallskip

\noindent {\bf Outline of the Construction in Theorem \ref{theoTreeConstruction}} To find this covering we define good and bad balls as follows: imagine our ball, $B$, has radius 1. We say that $B$ is a \emph{good ball}, if at every point in $x\in S^k_\varepsilon(u)\cap B$ the monotone quantity centered at that point at some small scale, $\rho$, is not much smaller than the monotone quantity on ball $B$ (we say these points have ``small density drop"). A ball $B$ is a \emph{bad ball} if all the points in $S^k_\varepsilon(u)\cap B$ with small density drop are contained in a small neighborhood of a $(k-1)$-plane. This good/bad is a dichotomy follows from Theorem \ref{theoKeyDichotomy} in Section \ref{secToolsBad}. 

In a good ball of radius $r$ we cover $S^k_\varepsilon(u)$ with balls of radius $\rho r$ iterating the construction until we find a bad ball or until the radius of the ball becomes very small. In a bad ball, we cover $S^k_\varepsilon(u)$ away from the $(k-1)$-plane without much care. Close to the $(k-1)$-plane we cover $S^k_\varepsilon(u)$ with balls of radius $\rho r$ iterating the construction until we reach a good ball or until the radius of the ball becomes very small. 

Inside long strings of good balls, the packing estimates follow from powerful tools in geometric measure theory (see Theorem \ref{theoDiscreteReifenberg} below) and the connection between the drop in monotonicity and the local flatness of the singular strata (see Theorem \ref{theoL2Estimate} below). We give more details in Section \ref{secToolsGood}. 

Inside  long strings of bad balls each of which is near the $(k-1)$-plane of the previous bad ball, we have even better packing estimates than expected (as we are effectively well approximated by planes which are lower dimensional). This leaves only points which are in many bad balls and in most of those balls they are far away from the $(k-1)$-plane. However, at these points the monotone quantity drops a definite amount many times, which contradicts either finiteness or monotonicity. This implies that the points and scales inside the bad balls which are not close to the $(k-1)$-plane form a negligible set (the technical term is a Carleson set).  We give more information about the bad balls in Section \ref{secToolsBad}. 

\subsection{Tools for bad balls: key dichotomy}\label{secToolsBad}

\begin{theorem}[Key dichotomy]\label{theoKeyDichotomy}
Let  $\epsilon, \rho, \gamma, \eta'>0$ be fixed numbers with $\rho \gamma<2$. There exists an
$\eta_0 (n, \alpha, \epsilon, \rho, \gamma, \eta')<\rho/100$ such that for every $\eta \leq \eta_0$, every $r>0$, every $E>0$ and every  minimizer $u\in \HH^\beta(B_{4r})$ of $\mathcal{J}$ in $B_{4r}$ with $0\in F(u)$ and $\sup_{B_r}\Psi^u_{2r}\leq E$,
then either
\begin{itemize}
\item $\Psi_{\gamma \rho r}^u\geq E-\eta'$ on $S^k_{\epsilon,\eta r}\cap B_r$, or
\item there exists $\ell \in L^{k-1}$ so that $\{x\in B_r: \Psi^u_{2\eta r}(x)\geq E-\eta\}\subset B_{\rho r}(\ell)$.
\end{itemize}
\end{theorem}

The key dichotomy is a direct consequence of the Lemma \ref{lemPreKey} below. The core idea is to make effective the following assertion: if $u$ is $k$-symmetric, then along the invariant manifold the Allen-Weiss density is constant, and every point away from the manifold will have $(k+1)$-symmetric blow-ups by Lemma \ref{lemConeBlowup}. 


\begin{lemma}\label{lemPreKey}
Let $\epsilon, \rho, \gamma, \eta'>0$ be fixed numbers with $\gamma\rho <2$. There exist $\eta_0, \theta>0$ such that for every $\eta<\eta_0$, every $E>0$ and every minimizer $u$ of $\mathcal{J}$ in $B_4$ with  $0\in F(u)$ and $\sup_{B_1}\Psi^u_2\leq E$, if  there exist $w_0,\dots, w_k\in B_1$ and affine manifolds $L^i:=\langle w_0,\dots, w_i\rangle\in V^{i}$ with 
$$w_i\notin B_\rho(L^{i-1}),\quad\quad  \mbox{ and } \quad\quad \Psi^u_{2\eta}(w_i) \geq E-\eta \quad\quad\mbox{for every }i\in\{0,\cdots,k\},$$
then, 
\begin{equation}\label{eqSmallJump}
\Psi^u_{\gamma\rho}(x)\geq E-\eta' \quad\quad \mbox{on }B_\theta(L^k)\cap B_1
\end{equation}
and
\begin{equation}\label{eqCloseStrata}
S^k_{\epsilon,\eta}\cap B_1\subset B_\theta(L^k)
\end{equation}
\end{lemma}

The proof follows (with only minor modifications) the proof in \cite[Lemma 3.3]{EdelenEngelstein}. We end this subsection by formally defining the good/bad balls alluded to above: 

\begin{definition}\label{defGoodBad}
Let $x\in B_2$, $0< R<r<2$ and $u$ be a minimizer to $\mathcal J$ in $B_5$. We say that the ball $B_r(x)$ is \emph{good} if 
$$\Psi^u_{\gamma\rho r}\geq E-\eta' \quad\quad \mbox{on } S^k_{\epsilon, \eta R}\cap B_r(x), $$
and  otherwise we say that $B_r(x)$ is \emph{bad}. 
\end{definition}
By Theorem \ref{theoKeyDichotomy} in any bad ball $B$ there exists an affine $(k-1)$-manifold $\ell_B$ with
\begin{equation}\label{eqBadEstimate}
\{w\in B: \Psi^u_{2\eta r}(w)\geq E-\eta\}\subset B_{\rho r}(\ell_B^{k-1}).
\end{equation}

\subsection{Tools for good balls: packing estimates and GMT}\label{secToolsGood}

In this section we control the local flatness of the singular strata by the drop in monotonicity. To do this we introduce a key tool from geometric measure theory which estimates the flatness of a set. Given a Borel measure $\mu$, a point $x$ and a radius $r$, the beta coefficient is defined as follows:
\begin{equation}\label{eqDefBeta}
\beta^k_{\mu,2}(B_r(x))^2:=\beta^k_{\mu,2}(x,r)^2 =\inf_{L \in V^k_a} \frac{1}{r^{k}} \int_{B_r(x)} \frac{\dist(z,L)^2}{r^2}\, d\mu (z)
\end{equation}
where $V^k_a$ stands for the collection of $k$-dimensional affine sets of $\R^n$. The beta coefficients are meant to measure in a scale invariant way how far is a measure from being flat, in this case in the $L^2$ distance, although other $L^p$ versions have been used in the literature for $1\leq p\leq\infty$ quite often, dating back to \cite{JonesTraveling} (for the $L^\infty$ version) and David-Semmes \cite{DavidSemmes} (for the $L^p$ version). 

If we control the size of the $\beta^k$'s we can conclude size and structure estimates on the measure $\mu$. The following theorem says exactly this and represents a major technical achievement. It differs (importantly) from prior work in this area by the lack of {\it a priori} assumptions on the upper or lower densities of the measure involved.

\begin{theorem}[Discrete-Reifenberg Theorem, see {\cite[Theorem 3.4]{NaberValtortaRectifiable}}]\label{theoDiscreteReifenberg}
Let $\{B_{r_q}(q)\}_q$ be a collection of disjoint balls, with $q\in B_1(0)$ and $0<r_q\leq1$, and let $\mu$ be the packing measure $\mu:=\sum_q r_q^k\delta_q$, where $\delta_q$ stands for the Dirac delta at $q$.  There exist constants $\tau_{DR}, C_{DR}>0$ depending only on the dimension such that if
$$\int_0^{2r}\int_{B_r(x)} \beta^k_{\mu,2}(z,s)^2 \, d\mu (z) \frac{ds}{s} \leq \tau_{DR} r^k \quad\quad \mbox{for every } x\in B_1(0), \, 0<r\leq 1,$$
then
$$\mu(B_1(0))=\sum_q r_q^k\leq C_{DR}.$$
\end{theorem}


To obtain the packing estimates required for the Discrete-Reifenberg Theorem, we need to control the beta coefficients. The key estimate of this entire framework lies in the following theorem, which shows the drop in monotonicity at a given point and a given scale controls the beta coefficient at a comparable scale.
\begin{theorem}\label{theoL2Estimate}
Let $\epsilon>0$ be given. There exist $\delta(n,\alpha, \epsilon)$ and $c(n, \alpha, \epsilon)$ such that  for every $u\in \HH^\beta(B_{5r})$ minimizing $\mathcal{J}$ in $B_{5r}(x)$ with  $x\in F(u)$ and
\begin{equation}\label{eq0DeltaNonK1Epsilon}
\begin{cases}
u \mbox{ is }(0,\delta)\mbox{-symmetric in } B_{4r}(x)\\
u \mbox{ is not }(k+1,\epsilon)\mbox{-symmetric in } B_{4r}(x),
\end{cases}
\end{equation}
and every Borel measure $\mu$, we have that
\begin{equation}\label{eqL2Estimate}
\beta^k_{\mu,2}(B_r(x))^2 
	\leq \frac{ c(n,\alpha, \epsilon)}{r^k}  \int_{B_r(x)} \left( \Psi_{4r}^u (w)-\Psi_{r}^u (w)\right) \, d\mu (w).
\end{equation}

\end{theorem}

We follow the proof of \cite[Theorem 5.1]{EdelenEngelstein} closely. First the authors give an explicit formula for the beta coefficients. 

\begin{lemma}\label{lemExplicitBetas}
Let $X$ be the center of mass of a  Borel measure $\mu$ on $B=B_r(x)$. Let $\{\lambda_i\}_{i=1}^n$ be the decreasing sequence of eigenvalues of the non-negative bilinear form
$$Q(v,w):=  \fint_{B}(v\cdot (z-X)) (w\cdot (z-X))\, d\mu (z),$$
and let $\{v_i\}_{i=1}^n$ be a corresponding orthonormal sequence of eigenvectors, that is $v_i\cdot v_j=\delta^{ij}$ and $Q(v_i , v)=\lambda_i v_i\cdot v$. Then
$$\beta_{\mu,2}^k(B)^2= \frac{1}{r^{k}} \int_{B} \frac{\dist(z,L^k)^2}{r^2}\, d\mu (z) = \frac{\mu(B)}{r^{k}}\frac{(\lambda_{k+1} + \dots + \lambda_n)}{r^2},$$
where $L^k := X+ \mathrm{span}\langle v_1,\dots, v_k \rangle$.
\end{lemma}

Next we find a relation between the eigenvalues of $Q$ and Allen-Weiss' energy.

\begin{lemma} \label{lemLambda2}
Under the hypothesis of Lemma \ref{lemExplicitBetas}, for every $u\in \HH^\beta(B_{5r})$ minimizing $\mathcal{J}$ in $B_{5r}(x)$ and every $i\leq n$, we have that
\begin{equation}\label{eqEigenAndMonotonicity}
\lambda_i \frac{2}{r^{n+2}}\int_{A_{2r,3r}(x)} |y|^\Beta (v_i\cdot Du(z))^2\, dz\leq C \fint_{B_r(x)} \left( \Psi_{4r}^u (w)-\Psi_{r}^u (w)\right) \, d\mu (w) .
\end{equation}
\end{lemma}
\begin{proof}
The argument follows as in \cite[(18) and below]{EdelenEngelstein}. In formula \emph{(18)} one needs to change $u(z)$ by $\alpha u(z)$, which can be done with exactly the same argument.
\end{proof}

Finally, using compactness, we bound the left-hand side of \rf{lemLambda2} from below.
\begin{lemma}\label{lemViBelow}
Let $\epsilon>0$ be given. There exists a $\delta(n,\alpha, \epsilon)$ and $c(n, \alpha, \epsilon)$ such that,  for  every orthonormal basis $\{v_i\}_{i=1}^n$ and every $u\in \HH^\beta(B_{5r})$ minimizing $\mathcal{J}$ in $B_{5 r}(x)$ with $x\in F(u)$ and satisfying \rf{eq0DeltaNonK1Epsilon}, we have that
\begin{equation}\label{eqViBelow}
\frac{1}{c(n, \alpha, \epsilon)} \leq r^{-n} \int_{A_{2r,3r}(x)} |y|^\Beta \sum_{i=1}^{k+1} (v_i\cdot Du(z))^2\, dz.\end{equation}
\end{lemma}
\begin{proof}
The proof follows that of \cite[(19)]{EdelenEngelstein} and we omit it.

\end{proof}

\begin{proof}[Proof of Theorem \ref{theoL2Estimate}]
By Lemmas \ref{lemExplicitBetas},  \ref{lemViBelow} and   \ref{lemLambda2} we get that
\begin{align*}
\beta^k_{\mu,2}(B)^2 
	&\leq \frac{\mu(B)}{r^{k+2}}(n-k) \lambda_{k+1} \\
	& \leq \frac{\mu(B)}{r^k} (n-k) c(n,\alpha, \epsilon)  \sum_{i=1}^{k+1}\frac{\lambda_i}{r^{n+2}}\int_{A_{2r, 3r}(x)} |y|^\Beta (v_i\cdot Du(z))^2 \, dz\\
	&\leq \frac{ c(n,\alpha, \epsilon)}{r^k}  \int_{B_r(x)} \left( \Psi_{4r}^u (w)-\Psi_{r}^u (w)\right) \, d\mu (w).
\end{align*}

\end{proof}

\renewcommand{\abstractname}{Acknowledgements}
\begin{abstract}
M.E. was partially supported by an NSF postdoctoral fellowship, NSF DMS 1703306 and by David Jerison's grant DMS 1500771. A.K. acknowledges Financial support from the Spanish Ministry of
Economy and Competitiveness, through the Mar\'ia de Maeztu Programme for Units of
Excellence in R\&D (MDM-2014-0445). M.P. was funded by the European Research Council under the grant agreement 307179-GFTIPFD. G.S. has received funding from the European Union's Horizon 2020 research and innovation programme under Marie Sk{\l}odowska-Curie grant agreement No 665919.  
A.K., M.P. and G.S. were also partially funded by 2017-SGR-0395 (Catalonia) and MTM-2016-77635-P (MINECO, Spain). 
Y.S. is partially supported by the Simons foundation.
	
M.E. would also like to thank Nick Edelen for many fruitful conversations regarding the quantitative stratification and rectifiable Reifenberg framework. 

A.K., M.P., G.S. would like to thank Xavier Cabr\'e, Tom\'as Sanz, Matteo Cozzi, Albert Mas, Maria del Mar Gonz\'alez, Luis Silvestre and Stefano Vita for some conversations around \cite{VitaThesis}. They would also like to thank Mihalis Mourgoglou for some conversations regarding the degenerate elliptic measure. 
\end{abstract}

\appendix
\section*{Appendix}
\section{Relation with the nonlocal Bernoulli problem}

As in \cite[Lemma 2.1]{DipierroValdinoci}, we see that the study of minimizers of $\mathcal{J}$ includes the study of minimizers of $J$.

\begin{proposition}\label{propoExtension}
If $f$ is a minimizer of $J$ in the unit ball of $\R^n$ then $f*P_y$ is a minimizer of $\mathcal{J}$ in every ball $B$ such that $B'\subset \subset B_1'$.

If $u=f*P_y$ is a minimizer of $\mathcal{J}$, then $f$ is a minimizer for $J$. In particular, if $u$ is a minimizer of $\mathcal{J}$ in every ball, positive outside the hyperplane $\{y=0\}$, and $u(x,y)=\mathcal{O}(|(x,y)|^\alpha)$,  then $u|_{\R^n\times\{0\}}$ is a minimizer for $J$ in every ball.
\end{proposition}

We follow \cite[Lemma 2.1]{DipierroValdinoci}, that is, we use the following result from \cite[Section 7]{CaffarelliRoquejoffreSavin}.
\begin{lemma}[see {\cite[Section 7]{CaffarelliRoquejoffreSavin}}]\label{lemCRS}
Let $f,g$ satisfy that $J_0(f,B_1), J_0(g,B_1) < \infty$, and suppose that $f-g$ is compactly supported in $B_1\subset\R^n$.
Then we have that
\begin{equation*}
J_0(g,B_1)  - J_0(f,B_1)
= c_{n,\alpha} \inf \int_{\Omega} |y|^\beta (|\nabla v(x,u)|^2 - |\nabla (f*P_y)(x)|^2),
\end{equation*}
where the infimum is taken among all the symmetric bounded Lipschitz domains $\Omega$ with the property that $\Omega\cap (\R^n\times\{0\})\subset B_1$ and among all symmetric functions $v$ with trace $g$ satisfying that $v-f*P_y$ is compactly supported on $\Omega$.
\end{lemma}

\begin{proof}[Proof of Proposition \ref{propoExtension}]
Let $f$ be a minimizer of $J$ in the unit ball of $\R^n$ and let $B_r$ be a ball such that $B_r'\subset \subset B_1'$. We want to show that $u:=f*P_y$ is a minimizer of $\mathcal{J}$ in $B_r$.

Let $v:\R^{n+1}\to \R$ so that $v\equiv u$ in $\R^{n+1}\setminus B_r$ and $v \in H^1(\beta, B_r)$. Let $g$ be the trace of $v$ in ${\R^n\times\{0\}}$.
By Lemma \ref{lemCRS} we have that
\begin{equation}\label{eqQuality}
J_0(g,B_1)  - J_0(f,B_1)\leq c_{n,\alpha} \int_{B_{r+\varepsilon}} |y|^\beta (|\nabla v|^2 - |\nabla u|^2)
\end{equation}
for every $\varepsilon>0$.

In particular, since $g|_{(B')^c }\equiv 0$,  $g$ is an admissible competitor for $f$
and
$J(f,B_1)\leq J(g,B_1)$, i.e.,
\begin{align}\label{eqFinestQuality}
J_0(g,B_1)  - J_0(f,B_1)    & \geq
-m(\{g>0\}\cap B_1) + m(\{f>0\}\cap B_1)\\
\nonumber & =  m(\{u>0\}\cap B_r')-m(\{v>0\}\cap B_r').
\end{align}
The proposition follows combining \rf{eqQuality} and \rf{eqFinestQuality} and letting $\varepsilon\to 0$.

The converse follows the same sketch: every global minimizer can be expressed as the Poisson extension of its restriction to the hyperplane by Proposition \ref{propoUniquenessDirichlet} and it is left to the reader.
\end{proof}

As a consequence of the previous proposition, all the results that we have proven for minimizers of $\mathcal{J}$ also apply to minimizers of $J$:
\begin{corollary}
If $u:\R^n\to \R$ is a minimizer to $J$ in $B_2\subset \R^n$ and $0\in F(u)$, then $\norm{u}_{C^\alpha(B_1)}\leq C$, it satisfies the nondegeneracy condition $u(x)\geq C \dist(x,F(u))^\alpha$ for $x\in B_1$, the positive phase satisfies the corkscrew condition, every blow-up limit is $\alpha$-homogeneous, and the boundary condition in \rf{eqEulerLagrangeFrac} is satisfied at $F_{\red}(u)$. 

Moreover, the positive phase $\{u>0\}\cap B_1$ is a set of  finite perimeter, the singular set is an $(n-3)$-rectifiable set, it is discrete whenever $n=3$ and it is empty if $n\leq 2$. 

All the constants depend only on $n$ and $\alpha$.
\end{corollary}

\section{Uniqueness of extensions}

In Proposition \ref{propoExtension} we have used the following result, included in \cite[Proposition 3.1]{CaffarelliRoquejoffreSire}. Here we provide a proof which is different than the one appearing in \cite{CaffarelliRoquejoffreSire}.

\begin{proposition}\label{propoUniquenessDirichlet}
Let $\alpha\in(0,1)$, $\beta=1-2\alpha$, and set $\mathcal{L}u=-\dive(|y|^{\beta}\nabla u)$ in $\mathbb R^{n+1}$. Suppose that $v:\overline{\mathbb R^{n+1}_+}\to\mathbb R$ is nonnegative outside $\mathbb R^n$, it is a solution to $\mathcal{L}v=0$ in $\mathbb R^{n+1}_+$ with $v(x',0)=0$ for all $x'\in\mathbb R^n$ and $|v(x)|\leq C|x|^{\alpha}$. Then $v\equiv 0$.
\end{proposition}
\begin{proof}
First, since $|y|^{\beta}$ is $C^{\infty}$ away from the hyperplane $\mathbb R^n$, $v\in C^{\infty}_{\loc}(\mathbb R^{n+1}_+)$. Let now $i\in\{1,\dots n\}$, and set
\[
f_m(x)=\frac{v\left(x+\frac{1}{m}e_i\right)-v(x)}{1/m}.
\]
Let $B_r=B_r(x',0)$ be a ball centered at $(x',0)\in\mathbb R^n \times\{0\}$ with radius $r$, and let $B_{2r}$ be its double ball. Set also $w(x)=w(x',y)=y^{\beta}$ for $y>0$. Since $f_m$ is a solution of $\mathcal{L}f_m=0$ in $B_r^+=B\cap\mathbb R^{n+1}_+$, \cite[Theorem 2.4.3]{FabesKenigSerapioni} shows that
\[
\max_{B_r^+}|f_m(x)|\leq C\left(\frac{1}{w(B_{2r}^+)}\int_{B_{2r}^+}|f_m|^2w\right)^{1/2}.
\]
From convergence of difference quotients (similarly to \cite[Theorem 3, page 277]{Evans}), if $v\in H^1(\beta,B_{2r}^+)$, the last estimate will imply that $f_m$ is uniformly bounded in $B_r^+$ by a constant $C_r$. Therefore, from the boundary Caccioppoli estimate  (\cite[(2.4.2)]{FabesKenigSerapioni}) we have that
\[
\int_{ B_{r/2}^+}|\nabla f_m|^2w\leq\frac{C}{r^2}\int_{ B_r^+}|f_m|^2w\leq \frac{C}{r^2}\int_{ B_r^+}C_r^2w\leq C_{n,r,w}<\infty,
\]
hence $\{f_m\}$ is bounded in $H^1(\beta,B_{r/2}^+)$. From weak compactness, a subsequence of $\{f_m\}$ converges to a solution of $\mathcal{L}u=0$ in $B_{r/2}^+$, and since $f_m\to\partial_i v$ pointwise, we obtain that $\partial_iv$ is  an $H^1(\beta,B_{r/2}^+)$ solution in $B_{r/2}^+$. Hence $\partial_iv$ is a solution to $\mathcal{L}u=0$ in $\mathbb R^{n+1}_+$.

Now, for $x=(x',y)\in\mathbb R^{n+1}_+$, let $R=|x|$. We distinguish between two cases: $y>R/16$, and $y<R/16$. 

In the first case, set $B_R$ to be the ball of radius $R$, centered at $x$. Note then that $B_{R/16}\subseteq\mathbb R^{n+1}_+$. Then, from \cite[Theorem 2.3.1]{FabesKenigSerapioni}, Caccioppoli's estimate and the assumption $|v(x)|\leq C|x|^{\alpha}$,
\begin{align*}
|\partial_iv(x)|{^2}&\leq \frac{C}{w(B_{R/32})}\int_{B_{R/32}}|\partial_i v|^2w\leq \frac{C}{w(B_{R/32})}\frac{C}{R^2}\int_{B_{R/16}}|v|^2w\leq \frac{C}{R^2}\frac{w(B_{R/16})}{w(B_{R/32})}\sup_{B_{R/16}}|v|\leq CR^{2\alpha-2}.
\end{align*}
In the second case, let $B_R$ be the ball centered at $(x',0)$ with radius $R$, and denote $B_R^+=B_R\cap\mathbb R^{N+1}_+$. Then $x\in B_{R/8}^+$, therefore from \cite[Theorem 2.4.3]{FabesKenigSerapioni} and the boundary Caccioppoli estimate,
\begin{align*}
|\partial_iv(x)|{^2}&\leq {\frac{C}{w(B_{R/8}^+)}\int_{B_{R/8}^+}|\partial_i v|^2w\leq\frac{C}{w(B_{R/8}^+)}\frac{C}{R^2}\int_{B_{R/4}^+}|v|^2w\leq \frac{C}{R^2}\frac{w(B_{R/4}^+)}{w(B_{R/8}^+)}\sup_{B_{R/4}^+}|v|\leq CR^{2\alpha-2}.}
\end{align*}
So, in all cases, $|\partial_i v(x)|\leq C|x|^{\alpha-1}$. Letting $R\to\infty$ and using the maximum principle, we find that $\partial_iv=0$ for any $i=1,\dots n$. Therefore $v$ does not depend on the first $n$ variables, { so} $v(x',y)=v(y)$. {Hence,} in $\mathbb R^{n+1}_+$,
\[
0=-\dive(y^{\beta} {\nabla} v(y))=-{ \partial_y}(y^{\beta}v'(y))\,\,\Rightarrow\,\, y^{\beta}v'(y)=\tilde{c},
\]
for some constant $\tilde{c}$. From \cite[Theorem 2.4.6]{FabesKenigSerapioni}, $v$ is H{\"o}lder continuous up to the boundary, therefore for any $y>0$,
\[
v(y)=v(y)-v(0)=\int_0^yv'=\int_0^y\tilde{c}s^{-\beta}\,ds=\frac{\tilde{c}}{1-\beta}y^{1-\beta},
\]
which implies that
\[
|\tilde{c}|=(1-\beta)y^{\beta-1}|v(y)|=(1-\beta)y^{\beta-1}|v(0,y)|\leq (1-\beta)y^{\beta-1}y^{\alpha}=(1-\beta)y^{-\alpha},
\]
for any $y>0$. Letting $y\to\infty$ we obtain that $\tilde{c}=0$, hence $v'(y)=0$ as well, which implies that $v$ is a constant. Since $v$ vanishes on $\mathbb R^n$, this implies that $v\equiv 0$.
\end{proof}

\bibliography{Llibres}

\begin{thebibliography}{CRS10b}

\bibitem[AC81]{AltCaffarelli}
Hans~Wilhem Alt and Luis~A. Caffarelli.
\newblock Existence and regularity for a minimum problem with free boundary.
\newblock {\em Journal f{\"u}r die reine und angewandte Mathematik},
  325:105--144, 1981.

\bibitem[AFP00]{AFP}
Luigi Ambrosio, Nicola Fusco, and Diego Pallara.
\newblock {\em Functions of bounded variation and free discontinuity problems}.
\newblock Oxford Mathematical Monographs. The Clarendon Press, Oxford
  University Press, New York, 2000.

\bibitem[All12]{Allen}
Mark Allen.
\newblock Separation of a lower dimensional free boundary in a two-phase
  problem.
\newblock {\em Mathematical research letters}, 19(5):1055--1074, 2012.

\bibitem[BV16]{BucurValdinoci}
Claudia Bucur and Enrico Valdinoci.
\newblock {\em Nonlocal diffusion and applications}, volume~20.
\newblock Springer, 2016.

\bibitem[CN13]{CheegerNaber}
Jeff Cheeger and Aaron Naber.
\newblock Lower bounds on ricci curvature and quantitative behavior of singular
  sets.
\newblock {\em Invent. Math.}, 191(2):321--339, 2013.

\bibitem[CRS10a]{CaffarelliRoquejoffreSavin}
Luis~A. Caffarelli, Jean-Michel Roquejoffre, and Ovidiu Savin.
\newblock Nonlocal minimal surfaces.
\newblock {\em Communications on Pure and Applied Mathematics},
  63(9):1111--1144, 2010.

\bibitem[CRS10b]{CaffarelliRoquejoffreSire}
Luis~A. Caffarelli, Jean-Michel Roquejoffre, and Yannick Sire.
\newblock Variational problems with free boundaries for the fractional
  laplacian.
\newblock {\em Journal of the European Mathematical Society}, 12(5):1151--1179,
  2010.

\bibitem[CS05]{CaffarelliSalsa}
Luis~A. Caffarelli and Sandro Salsa.
\newblock {\em A geometric approach to free boundary problems}, volume~68.
\newblock American Mathematical Society Providence, RI, 2005.

\bibitem[CS07]{CaffarelliSilvestre}
Luis~A. Caffarelli and Luis Silvestre.
\newblock An extension problem related to the fractional laplacian.
\newblock {\em Communications in partial differential equations},
  32(8):1245--1260, 2007.

\bibitem[CS14]{cabreSire}
Xavier Cabr{\'e} and Yannick Sire.
\newblock Nonlinear equations for fractional {L}aplacians, {I}: {R}egularity,
  maximum principles, and {H}amiltonian estimates.
\newblock {\em Ann. Inst. H. Poincar{\'e} Anal. Non Lin{\'e}aire},
  31(1):23--53, 2014.

\bibitem[{Dal}12]{DalMaso}
Gianni {Dal Maso}.
\newblock {\em An introduction to {$\Gamma$}-convergence}, volume~8.
\newblock Springer Science \&amp; Business Media, 2012.

\bibitem[DJ09]{DSJ}
Daniela {De Silva} and David Jerison.
\newblock A singular energy minimizing free boundary.
\newblock {\em J. Reine Angew. Math.}, 635:1--21, 2009.

\bibitem[DL76]{duvautLions}
Georges Duvaut and Jacques-Louis Lions.
\newblock {\em Inequalities in mechanics and physics}.
\newblock Springer-Verlag, Berlin-New York, 1976.
\newblock Translated from the French by C. W. John, Grundlehren der
  Mathematischen Wissenschaften, 219.

\bibitem[DR12]{desilvaRoque}
Daniela {De Silva} and Jean-Michel Roquejoffre.
\newblock Regularity in a one-phase free boundary problem for the fractional
  {L}aplacian.
\newblock {\em Ann. Inst. H. Poincar{\'e} Anal. Non Lin{\'e}aire},
  29(3):335--367, 2012.

\bibitem[DS93]{DavidSemmes}
Guy David and Stephen Semmes.
\newblock {\em Analysis of and on uniformly rectifiable sets}, volume~38.
\newblock American Mathematical Soc., 1993.

\bibitem[DS12]{DSSJDE}
Daniela {De Silva} and Ovidiu Savin.
\newblock {$C^{2,\alpha}$} regularity of flat free boundaries for the thin
  one-phase problem.
\newblock {\em J. Differential Equations}, 253(8):2420--2459, 2012.

\bibitem[DS15]{deSilvaSavinJEMS}
Daniela {De Silva} and Ovidiu Savin.
\newblock Regularity of lipschitz free boundaries for the thin one-phase
  problem.
\newblock {\em Journal of the European Mathematical Society}, 17(6):1293--1326,
  2015.

\bibitem[DSFS19]{DeSilvaFerrariSalsaSurvey}
Daniela De~Silva, Fausto Ferrari, and Sandro Salsa.
\newblock Recent progresses on elliptic two-phase free boundary problems.
\newblock to appear in Disc. Cont. Dyn. Systems, 2019.

\bibitem[DSS14]{deSilvaSavinSire}
Daniela {De Silva}, Ovidiu Savin, and Yannick Sire.
\newblock A one-phase problem for the fractional laplacian: regularity of flat
  free boundaries.
\newblock {\em arXiv preprint arXiv:1401.6443}, 2014.

\bibitem[DT15]{davidtoroalmostminimizers}
Guy David and Tatiana Toro.
\newblock Regularity of almost minimizers with free boundary.
\newblock {\em Calc. Var. Partial Differential Equations}, 54(1):455--524,
  2015.

\bibitem[DV17]{DipierroValdinoci}
Serena Dipierro and Enrico Valdinoci.
\newblock Continuity and density results for a one-phase nonlocal free boundary
  problem.
\newblock In {\em Annales de l'Institut Henri Poincare (C) Non Linear
  Analysis}, volume~34, pages 1387--1428. Elsevier, 2017.

\bibitem[EE19]{EdelenEngelstein}
Nick Edelen and Max Engelstein.
\newblock Quantitative stratification for some free-boundary problems.
\newblock {\em Trans. Amer. Math. Soc.}, 371(3):2043--2072, 2019.

\bibitem[EG15]{EvansGariepy}
Lawrence~Craig Evans and Ronald~F. Gariepy.
\newblock {\em Measure theory and fine properties of functions}.
\newblock CRC press, 2015.

\bibitem[Eva98]{Evans}
Lawrance~Craig Evans.
\newblock {\em Partial differential equations}, volume~19 of {\em Graduate
  Studies in Mathematics}.
\newblock Oxford University Press, 1998.

\bibitem[Fed69]{Federer}
Herbert Federer.
\newblock {\em Geometric measure theory}.
\newblock Springer, 1969.

\bibitem[FJK82]{FabesJerisonKenig}
Eugene~Barry Fabes, David Jerison, and Carlos~E. {Kenig}.
\newblock The wiener test for degenerate elliptic equations.
\newblock {\em Ann. Inst. Fourier (Grenoble)}, 32(3):151--182, 1982.

\bibitem[FKS82]{FabesKenigSerapioni}
Eugene~Barry Fabes, Carlos~E. Kenig, and Raul~Paolo Serapioni.
\newblock The local regularity of solutions of degenerate elliptic equations.
\newblock {\em Comm. Partial Differential Equations}, 7(1):77--116, 1982.

\bibitem[FV17]{figalliValdinoci}
Alessio Figalli and Enrico Valdinoci.
\newblock Regularity and {B}ernstein-type results for nonlocal minimal
  surfaces.
\newblock {\em J. Reine Angew. Math.}, 729:263--273, 2017.

\bibitem[HKM06]{HeinonenKilpelainenMartio}
Juha Heinonen, Tero Kilpel{\"a}inen, and Olli Martio.
\newblock {\em Nonlinear potential theory of degenerate elliptic equations}.
\newblock Courier Corporation, 2006.

\bibitem[Jon90]{JonesTraveling}
Peter~W. Jones.
\newblock Rectifiable sets and the traveling salesman problem.
\newblock {\em Invent. Math.}, 102(1):1--15, 1990.

\bibitem[KT03]{KenigToroDCDS}
Carlos~E. {Kenig} and Tatiana {Toro}.
\newblock On the free boundary regularity theorem of alt and caffarelli.
\newblock {\em Discrete and Continuous Dynamical Systems}, 10:397--422, 2003.

\bibitem[NV17]{NaberValtortaRectifiable}
Aaron Naber and Daniele Valtorta.
\newblock Rectifiable-reifenberg and the regularity of stationary and
  minimizing harmonic maps.
\newblock {\em Annals of Mathematics}, pages 131--227, 2017.

\bibitem[Sil12]{Silvestre}
Luis Silvestre.
\newblock On the differentiability of the solution to an equation with drift
  and fractional diffusion.
\newblock {\em Indiana University Mathematics Journal}, pages 557--584, 2012.

\bibitem[SV12]{savinValdinoci}
Ovidiu Savin and Enrico Valdinoci.
\newblock {$\Gamma$}-convergence for nonlocal phase transitions.
\newblock {\em Ann. Inst. H. Poincar{\'e} Anal. Non Lin{\'e}aire},
  29(4):479--500, 2012.

\bibitem[SV13]{SVcone}
Ovidiu Savin and Enrico Valdinoci.
\newblock Regularity of nonlocal minimal cones in dimension 2.
\newblock {\em Calc. Var. Partial Differential Equations}, 48(1-2):33--39,
  2013.

\bibitem[Vit18]{VitaThesis}
Stefano Vita.
\newblock {\em Strong competition systems ruled by anomalous diffusions}.
\newblock PhD thesis, Universit{\'a} degli Studi di Torino, 2018.

\bibitem[Wei98]{WeissWeak}
Georg~Sebastian Weiss.
\newblock Partial regularity for weak solutions of an elliptic free boundary
  problem.
\newblock {\em Communications in partial differential equations},
  23(3-4):439--455, 1998.

\bibitem[Wei99]{WeissMinimum}
Georg~Sebastian Weiss.
\newblock Partial regularity for a minimum problem with free boundary.
\newblock {\em Journal of Geometric Analysis}, 9(2):317--326, 1999.

\end{thebibliography}
\end{document}